\documentclass{amsart}
\pdfoutput=1
\usepackage{float}
\floatstyle{boxed}
\restylefloat{figure}
\usepackage{mathtools}
\usepackage{amssymb}
\usepackage{mathrsfs}
\usepackage[all]{xy}
\UseComputerModernTips
\CompileMatrices
\usepackage{enumitem}
\usepackage{graphicx} 
\usepackage{tikz}
\usepackage[colorlinks,allcolors=blue]{hyperref}
\usepackage{ifthen}
\usepackage{latexsym}
\allowdisplaybreaks
\numberwithin{equation}{section}
\newtheorem{theorem}[equation]{Theorem}

\newtheorem{proposition}[equation]{Proposition}
\newtheorem{corollary}[equation]{Corollary}
\newtheorem*{namedthm}{Theorem \namedthmname}
\newcounter{namedthm}

\makeatletter
\newenvironment{namedtheorem}[1]
  {\def\namedthmname{#1}%
   \refstepcounter{namedthm}%
   \namedthm\def\@currentlabel{#1}}
  {\endnamedthm}
\makeatother
\theoremstyle{definition}
\newtheorem{definition}[equation]{Definition}

\theoremstyle{remark}
\newtheorem{remark}[equation]{Remark}

\renewcommand{\phi}{\varphi}
\DeclareMathSymbol{\boxprod}{\mathbin}{AMSa}{"03} % binary operator version of \square
\DeclareMathSymbol{\mixprod}{\mathbin}{AMSa}{"4F} % binary operator version of \triangledown

\newcommand{\dirsum}{\oplus}

\newcommand{\disjunion}{\sqcup}

\newcommand{\dual}{^\vee}

\newcommand{\homeo}{\approx}

\newcommand{\iso}{\cong}

\newcommand{\Mackey}[1]{\underline{#1}\vphantom{#1}}

\newcommand{\smsh}{\wedge}

\newcommand{\susp}{\Sigma}
\newcommand{\tensor}{\otimes}

\newcommand{\C}{{\mathbb C}}

\newcommand{\PP}{\mathbb{P}}

\newcommand{\R}{{\mathbb R}}

\newcommand{\sX}{{\mathcal X}}

\newcommand{\Z}{\mathbb{Z}}
\newcommand{\rels}[1]{\left\langle #1 \right\rangle}
\newcommand{\eval}[2]{\left\langle #1, #2 \right\rangle}

\DeclareMathOperator{\Hom}{Hom}

\newcommand{\ROev}{RO_{\mathrm{e}}}
\newcommand{\GG}{{C_2}}
\newcommand{\tensorS}{\tensor_{\HS}}
\newcommand{\HS}{{\mathbb H}}

\newcommand{\tE}{\tilde E}

\newcommand{\cw}{c_{\omega}}
\newcommand{\cxw}{c_{\chiw}}
\newcommand{\cl}{c_{\lambda}}
\newcommand{\cxl}{c_{\chi\lambda}}

\newcommand{\cwd}[1][]{\widehat{c}_{\omega}^{\ifthenelse{\equal{#1}{}}{}{{\:#1}}}}
\newcommand{\cxwd}[1][]{\widehat{c}_{\chi\omega}^{\ifthenelse{\equal{#1}{}}{}{{\:#1}}}}
\newcommand{\cld}[1][]{\widehat{c}_{\lambda}^{\ifthenelse{\equal{#1}{}}{}{{\:#1}}}}
\newcommand{\cxld}[1][]{\widehat{c}_{\chi\lambda}^{\ifthenelse{\equal{#1}{}}{}{{\:#1}}}}

\newcommand{\clod}[1][]{\widehat{c}_{\omega_1}^{\ifthenelse{\equal{#1}{}}{}{{\:#1}}}}
\newcommand{\cxlod}[1][]{\widehat{c}_{\chi\omega_1}^{\ifthenelse{\equal{#1}{}}{}{{\:#1}}}}
\newcommand{\cltd}[1][]{\widehat{c}_{\omega_2}^{\ifthenelse{\equal{#1}{}}{}{{\:#1}}}}
\newcommand{\cxltd}[1][]{\widehat{c}_{\chi\omega_2}^{\ifthenelse{\equal{#1}{}}{}{{\:#1}}}}
\newcommand{\cltensd}[1][]{\widehat{c}_{\omega_1\tensor\omega_2}^{\ifthenelse{\equal{#1}{}}{}{{\:#1}}}}
\newcommand{\cxltensd}[1][]{\widehat{c}_{\chi\omega_1\tensor\omega_2}^{\ifthenelse{\equal{#1}{}}{}{{\:#1}}}}
\newcommand{\cd}[1][]{\widehat{c}^{\ifthenelse{\equal{#1}{}}{}{{\:#1}}}}
\newcommand{\Cpq}[2]{\C^{#1+#2\sigma}}
\newcommand{\Cq}[1]{\C^{#1\sigma}}
\newcommand{\Xpq}[2]{\PP(\Cpq{#1}{#2})}
\newcommand{\Xp}[1]{\PP(\C^{#1})}
\newcommand{\Xq}[1]{\PP(\Cq{#1})}

\newcommand{\chiw}{\chi\omega}
\DeclareMathOperator{\Gr}{Gr}
\DeclareMathOperator{\grad}{grad}
\newcommand{\gr}{\Diamond}      % Symbol for common grading ([black]lozenge, bigstar, bullet are other possibilities)
\newcommand{\ext}{\mathsf{\Lambda}}     % possible alternative: {\bigwedge\nolimits}
\begin{document}

\title
%    [The $\GG$-equivariant ordinary cohomology of $BU(2)$]
%    {The $RO(\Pi BU(2))$-graded $\GG$-equivariant ordinary cohomology of $B_\GG U(2)$}
    {The $\GG$-equivariant ordinary cohomology of $BU(2)$}

\author{Steven R. Costenoble}
\address{Steven R. Costenoble\\Department of Mathematics\\Hofstra University\\
  Hempstead, NY 11549, USA}
\email{Steven.R.Costenoble@Hofstra.edu}
\author{Thomas Hudson}
\address{Thomas Hudson, College of Transdisciplinary Studies, DGIST, 
Daegu, 42988, Republic of Korea}
\email{hudson@dgist.ac.kr}

%
%\date{November 2024}
%
\abstract
We calculate the ordinary $\GG$-cohomology, with Burnside ring coefficients, of
$BU(2)$, the classifying space for $\GG$-equivariant complex 2-plane bundles,
using an extended grading that allows us to capture a more natural set of generators.
This allows us to define characteristic classes for such bundles.
Combined with earlier calculations, it also allows us to define characteristic numbers
for equivariant complex lines and surfaces and we give some sample computations.
\endabstract
\maketitle
\tableofcontents

\section{Introduction}

In nonequivariant algebraic topology, characteristic classes are fundamental tools
in the study of vector bundles and manifolds. 
Among these are the Chern classes, invariants of complex vector bundles, 
whose definition is closely linked to the computation of the cohomology of the classifying spaces $BU(n)$, 
the infinite Grassmannians of $n$-dimensional planes in $\C^\infty$, also denoted $\Gr_n(\C^\infty)$. 
In this classical setting, we have
\begin{equation}\label{classicBUn}
H^*(BU(n)_+;\Z)\iso \Z[c_1,...,c_n]\,,
\end{equation}
where the generator $c_i$ lives in grading $2i$ and is the $i$th Chern class of the tautological bundle $\omega(n)$ over $BU(n)$. 
For any rank $n$ bundle $E\to B$, the Chern classes of $E$ are then obtained by pulling back the generators 
along the classifying map $B\to BU(n)$.

A natural question to ask is how (\ref{classicBUn}) has to be modified when one considers 
topological spaces and vector bundles endowed with the action of a group $G$. 
In other words, how can we define $G$-equivariant characteristic classes? 
Inherent in this question is another: In what equivariant cohomology theory should these classes live? 

Our opinion is that that the natural environment for $G$-equivariant characteristic classes
is the equivariant ordinary cohomology with expanded grading
developed for this purpose by the first author and Stefan Waner in  \cite{CostenobleWanerBook}.
Before discussing the reasons behind this choice, we present our main theorem, which shows how (\ref{classicBUn})
translates to equivariant ordinary cohomology in the specific case of $BU(n)$ for $n=2$ and $G=\GG$. 
The case $n=1$ was studied in \cite{Co:InfinitePublished}, in which the first author, employing the Euler classes defined in \cite{CostenobleWanerBook}, described the characteristic classes of $\GG$-equivariant line bundles.

In order to state our main result we need to introduce a bit of notation. 
We denote by $\C$ the trivial complex representation of $\GG$ and by $\C^\sigma$ the sign representation. 
As a model of $BU(2)$ we shall use $\Gr_2(\C^\infty\oplus \C^{\infty\sigma})$, 
the Grassmannian of $2$-dimensional complex planes in $\C^\infty\oplus \C^{\infty\sigma}$, with the action of $\GG$ induced
by its action on $\C^\infty\oplus \C^{\infty\sigma}$. 
Over $BU(2)$ we have $\omega$, the tautological vector bundle of rank 2,  and its determinant line bundle $\lambda:=\ext^2 \omega$. 
Both can be tensored with $\C^\sigma$, giving bundles $\chi\omega$ and $\chi\lambda$, respectively. 
For any bundle $V$, we will denote by $c_V$ its Euler class. 
Finally, we denote by $\HS$ the $RO(\GG)$-graded equivariant ordinary cohomology of a point with Burnside ring coefficients
and by $ H_\GG^\gr(B U(2)_+)$ 
the cohomology of $BU(2)$ graded on $RO(\Pi BU(2)) \iso \Z^4$.
(See Appendix~\ref{app:ordinarycohomology} for more about $\HS$ and $H_\GG^\gr$.)

The following is our $C_2$-equivariant analogue of (\ref{classicBUn}) when $n=2$. 

\vspace{0.2 cm}

\begin{namedtheorem}{A}\label{thm:A}
As a module, $ H_\GG^\gr(B U(2)_+)$ is free over $\HS$, and as commutative algebra we have
\begin{equation}\label{eqBU2}
H_\GG^\gr(B U(2)_+)\iso \HS[\zeta_0,\zeta_1,\zeta_2,\cl,\cxl,\cw,\cxw]/I\,,
\end{equation}
where $I$ is the ideal generated by the relations
\begin{align*}
	\zeta_0\zeta_1\zeta_2 &= \xi \\
	\zeta_1 \cxl  &= (1-\kappa)\zeta_0\zeta_2 \cl + e^2 \\
	\zeta_2^2 \cxw &= (1-\kappa)\zeta_0^2 \cw  + e^2 \cxl \,.
\end{align*}
\qed
\end{namedtheorem}

\vspace{0.2 cm}
Thus, for any $\GG$-equivariant complex 2-plane bundle $E\to B$, we have characteristic classes
$\cw(E)$, $\cxw(E)$, and so on, in the cohomology of $B$,
given by pulling back elements of the cohomology of $BU(2)$ along a classifying map
for $E$.
The elements $\zeta_i$ are elements that come from the component structure of the fixed points $BU(2)^\GG$
and are described in detail in \cite[\S16]{Co:InfinitePublished}, while $\xi$, $\kappa$, and $e$ are elements of $\HS$.

Note that the calculation (\ref{eqBU2}) specializes to (\ref{classicBUn}) on forgetting the $\GG$-action,
since $\zeta_0$, $\zeta_1$ and $\zeta_2$ restrict to 1 nonequivariantly, the classes $\cl$ and $\cxl$ restrict to $c_1$, and
$\cw$ and $\cxw$ restrict to $c_2$. 
The three relations become trivial because $\xi$ restricts to 1 while $\kappa$ and $e^2$ restrict to 0.

One important use of nonequivariant characteristic classes is in defining characteristic numbers of manifolds,
which are invariants of bordism.
As an example of a possible application of our equivariant characteristic classes, we shall use Theorem~\ref{thm:A},
together with earlier calculations in \cite{CHTFiniteProjSpace},
to compute characteristic numbers for some complex lines and surfaces,
showing that, at least in these cases, 
these equivariant characteristic numbers are sufficient to distinguish the manifolds' equivariant bordism classes.

We promised earlier to motivate our choice of equivariant ordinary cohomology with expanded grading
as the natural framework for the definition of characteristic classes of vector bundles.
For this, it is probably useful to say a few words about the alternative choices that we have decided not to use.

The best known equivariant generalization of ordinary cohomology is probably Borel cohomology,
which has many useful features, including ease of calculation. For example, with $BH$ denoting Borel cohomology, we have that
\begin{align*}
    BH_\GG^*(BU(2)_+;\Z) &\iso BH_\GG^*(S^0;\Z)\tensor H^*(BU(2)_+;\Z) \\
    &\iso H^*(\GG;\Z)\tensor H^*(BU(2)_+;\Z).
\end{align*}
But this is too blunt an instrument for our purposes.
For example, in \cite{CHTAlgebraic}, we gave an equivariant version of B\'ezout's theorem using Borel cohomology
and showed that it does not give us any information about fixed sets, whereas 
the Euler class computed in our preferred ordinary cohomology encodes information
that completely determines both the cardinality and location of the fixed sets of intersections of hypersurfaces
in equivariant projective space.
The underlying problem is that Borel cohomology is an ``invariant'' theory, using the term
introduced by Peter May in \cite{May:Borel}. That means that, if $f\colon X\to Y$ is a $G$-map
such that $f$ is a nonequivariant equivalence, not necessarily a $G$-equivalence, 
then $f$ induces an isomorphism in Borel cohomology.
May gave a complete description of the characteristic classes in such theories,
including the general version of the calculation above, with the conclusion that
``no such theory is powerful enough to support a very useful theory of characteristic classes.''

Another option is $RO(G)$-graded ordinary (Bredon) cohomology, 
which was recently used by Hill, Hopkins, and Ravenel in \cite{HHRKervaire} to prove the Kervaire invariant one conjecture. 
Calculations related to the one we do here have been done in the $RO(G)$-graded theory,
notably those by Lewis \cite{LewisCP}; more recent examples include work by Dugger
\cite{DuggerGrass}, Kronholm \cite{KronholmSerre}, Hogle \cite{Hogle} and Hazel \cite{HazelFundamental,HazelSurfaces}.
But the most basic characteristic class we would expect to have is the Euler class, and the only vector bundles with Euler
classes in $RO(G)$-graded ordinary cohomology are those for which there is a single representation $V$ of $G$
such that every fiber of the bundle is a copy of $V$ (restricted to the isotropy group of the point of the base
over which the fiber sits).
Thus, many interesting vector bundles, including the tautological bundle over $BU(2)$,
fail to have Euler classes in this theory.
For this reason, $RO(G)$-graded ordinary cohomology
is also too limited to host a satisfactory theory of characteristic classes and needs to be extended. 

That was one of the reasons behind \cite{CostenobleWanerBook}, in which the first author and Stefan Waner developed 
equivariant ordinary cohomology with an extended grading.
From now on, when we refer to ``equivariant ordinary cohomology,'' this is the theory we mean.
The extended grading is exactly what is needed to allow us to define Euler classes for all vector bundles.
%In \cite{Co:InfinitePublished}, the first author used this extended grading to calculate the cohomology of
%$B_\GG U(1)$, the classifying space for complex $\GG$-line bundles, which gives the
%characteristic classes for such bundles.

The extension of the grading comes at a considerable increase in complexity. 
$RO(\GG)$-graded cohomology has a grading with two indices regardless of the space involved, 
because $RO(\GG) \iso \Z^2$, so is often referred to as a \emph{bi-graded} theory. 
The extended grading, though, depends on the space itself and more specifically on the topological features of its fixed points. 
The more complex the structure of the fixed points, the more complicated the grading will become. 
There are, however, considerable advantages to be reaped in compensation and the resulting
calculations may actually end up being simpler.

We invite the skeptical reader to compare Lewis's remarkable (and yet quite mysterious) computations of the 
$RO(\GG)$-graded cohomology of finite projective spaces in \cite{LewisCP} to the
significantly simpler description of their equivariant ordinary cohomology with extended grading 
we (together with Sean Tilson) obtained in \cite{CHTFiniteProjSpace}. 
In particular, \cite[Section 5.2]{CHTFiniteProjSpace} gives a comparatively more homogeneous description
of the $RO(\GG)$-graded part.
The $RO(\GG)$-graded calculation is just a slice of the complete picture, one that does not contain
the natural generators, hence is more complicated to describe without the proper context.

Another advantage is the existence of push-forward maps, which allowed us in \cite{CHTAlgebraic} and \cite{CH:geometric} 
to provide a geometric interpretation of what would otherwise have been just an algebraic picture.
In particular, we showed that every additive generator of the cohomology of an equivariant projective space
could be viewed as the fundamental class of singular manifold. 
As a byproduct, we obtained a refinement of the classical B\'ezout theorem for the intersection of an arbitrary number of algebraic hypersurfaces.

Another comment about our approach:
Nonequivariantly, the cohomology of $BU(2)$ is often derived from the cohomology of $BT^2$, which is obtained
easily from the cohomology of $BT^1 = BU(1)$. Equivariantly, the relations between these cohomologies are
more complicated, so here we calculate the cohomology of $BU(2)$ directly.
We will explore the relations between these cohomologies in \cite{CH:bt2},
which will concentrate on $BT^2$.

In future work we hope to use the calculations done here to look at the cohomology of finite
Grassmannians. In \cite{CHTFiniteProjSpace}, we used the cohomology of $BU(1)$ in calculating
the cohomologies of finite projective spaces and showed that they are not simply quotients
of the cohomology of the infinite case.
Preliminary work on finite Grassmannians shows that the relation between their cohomologies
and the cohomology of $BU(2)$ may be yet more complicated.

Of course, we would also like to know the cohomology of $BU(n)$ for every $n$,
but we do not yet have the machinery to tackle that general problem.
One difficulty is that, in $BU(1)$ and $BU(2)$, just as nonequivariantly,
the generators can be understood as Euler classes, so we have ready candidates.
However, even nonequivariantly, the Chern classes $c_i$ for $1 < i < n$ are not
Euler classes of bundles over $BU(n)$, so we will need other ways of identifying
generators for the equivariant cohomology when $n>2$.

We will assume some familiarity with ordinary cohomology graded on representations
of the fundamental groupoid of a space, but Appendix~\ref{app:ordinarycohomology} contains
a brief summary and references;
see particularly the recent guide by Beaudry et al., \cite{Beaudry:Guide}.
Because we will need to use rings graded on several different grading groups, we adopt the following notation:
If $A$ is a ring graded on $R$, we write $A^R$ to denote $A$ with that grading.
Write $RO(\Pi BU(2))$ for the representation ring of the fundamental groupoid of $BU(2)$.
Although our usual notation for the cohomology of $BU(2)$ graded on $RO(\Pi BU(2))$ is
$H_\GG^{RO(\Pi BU(2))}(BU(2)_+)$, we will simplify and write $H_\GG^\gr(BU(2)_+)$,
and we will use a superscript $\gr$ throughout this paper to indicate grading specifically on $RO(\Pi BU(2))$.
Where other grading is being used, we will give it explicitly.

The paper is structured as follows. 
We begin, in the following section, by proving that, in ``even'' gradings,
the cohomology of $BU(n)$ injects into the cohomology of its fixed sets,
which we compute in \S\ref{sec:fixed sets}.
Computations are easier there and we use this to verify that the relations listed
in Theorem~\ref{thm:A} actually hold.
In \S\ref{sec:freeness}, we show that the algebra proposed in Theorem~\ref{thm:A}
is free as a module over $\HS$, with tedious parts of the proof relegated to Appendix~\ref{app:resolution}.
Because we need to relate several graded rings with different grading groups,
in \S\ref{sec:change of grading} we briefly discuss what happens when we perform a change of grading to make them match,
which, in our situations, amounts to adjoining some new invertible elements.
We finally prove Theorem~\ref{thm:A} in \S\ref{sec:mainresult},
restated as Theorem~\ref{thm:main}.
The method of proof is one we used in \cite{CHTFiniteProjSpace}, based on a result that
says that, if a set of elements restricts to a basis both nonequivariantly and on taking fixed points,
then it is a basis over $\HS$, an approach suggested to us by the referee of that paper.
We apply this to the elements from a basis of the algebra that we proved free in \S\ref{sec:freeness}.
Following that, in \S\ref{sec:bases}, we give several examples of how the basis elements
are arranged and how they relate to nonequivariant bases.
In \S\ref{sec:units}, we find the units in the cohomology ring of $BU(2)$ and use them in
calculating the Euler classes of duals of bundles.
Finally, in \S\ref{sec:char numbers}, we give some example computations of characteristic numbers
of complex lines and surfaces, combining the computations from this paper with those
from~\cite{CHTFiniteProjSpace}.

\subsection*{Acknowledgements} 
Both authors would like to thank Sean Tilson for his help
in laying the foundations of this work. The first author thanks Hofstra University for released time to work on this project.
The second author was partially supported by the National Research Foundation of Korea (NRF) grant funded by the Korea government (MSIT) (No. RS-2024-00414849).

\section{A useful monomorphism}

Let $E\GG$ be the universal free $\GG$-space and let $\tE\GG$ be the based cofiber in the sequence
\[
    (E\GG)_+\to S^0 \to \tE\GG
\]
where the first map is induced by the projection $E\GG\to *$ to a point.
The $RO(\GG)$-graded cohomologies of $E\GG$ and $\tE\GG$ were computed and used
extensively in \cite{Co:InfinitePublished}.
The following result applies to all of the $\GG$-spaces $BU(n)$ and $BT^n$
and was used in \cite{Co:InfinitePublished} for $BU(1)$.

\begin{theorem}\label{thm:diagram}
Suppose that $B$ is a $\GG$-space and suppose that $B^\GG$ has a component $B^0$
such that the inclusion $B^0\to B$ is a nonequivariant equivalence. Then we have the
following commutative diagram with exact rows and columns and
with the indicated isomorphisms, monomorphisms, and epimorphisms.
(In this diagram, we write $R = RO(\Pi B)$ for brevity.)
% \[
% \def\objectstyle{\scriptstyle}
%  \xymatrix@C-0.5em{
%     H_\GG^R(B_+\smsh_B\tE\GG) \ar[r]^\eta_\iso \ar[d]_{\psi}
%      &  H_\GG^R(B_+^\GG\smsh_B\tE\GG) \ar[d]^\psi \\
%     H_\GG^R(B_+) \ar[r]^\eta \ar[d]_\phi
%      &  H_\GG^R(B_+^\GG) \ar[r]^-\theta \ar[d]^\phi
%      & \susp^{-1} H_\GG^{R}(B/_B B^\GG) \ar[d]^\phi_\iso \\
%     H_\GG^R(B_+\smsh_B (E\GG)_+) \ar@{>->}[r]^\eta \ar[d]_\delta
%      &  H_\GG^R(B_+^\GG\smsh_B (E\GG)_+) \ar@{->>}[r]^-\theta \ar[d]^\delta
%      & \susp^{-1} H_\GG^{R}(B/_B B^\GG\smsh_B (E\GG)_+) \\
%    \susp^{-1} H_\GG^{R}(B_+\smsh_B\tE\GG) \ar[r]^\eta_\iso
%      & \susp^{-1} H_\GG^{R}(B_+^\GG\smsh_B\tE\GG)
%   }
% \]
\[
\def\objectstyle{\scriptstyle}
 \xymatrix@C-0.5em{
    H_\GG^R(B_+\smsh_B\tE\GG) \ar[r]^\eta_\iso \ar[d]
     &  H_\GG^R(B_+^\GG\smsh_B\tE\GG) \ar[d] \\
    H_\GG^R(B_+) \ar[r]^\eta \ar[d]
     &  H_\GG^R(B_+^\GG) \ar[r] \ar[d]
     & \susp^{-1} H_\GG^{R}(B/_B B^\GG) \ar[d]^\iso \\
    H_\GG^R(B_+\smsh_B (E\GG)_+) \ar@{>->}[r]^\eta \ar[d]
     &  H_\GG^R(B_+^\GG\smsh_B (E\GG)_+) \ar@{->>}[r] \ar[d]
     & \susp^{-1} H_\GG^{R}(B/_B B^\GG\smsh_B (E\GG)_+) \\
   \susp^{-1} H_\GG^{R}(B_+\smsh_B\tE\GG) \ar[r]^\eta_\iso
     & \susp^{-1} H_\GG^{R}(B_+^\GG\smsh_B\tE\GG)
  }
\]
\end{theorem}

\begin{proof}
The rows are exact because they are parts of the long exact sequences induced by the pair
$(B,B^\GG)$.
The columns are exact because they are parts of the long exact sequences induced by the cofibration
$(E\GG)_+\to S^0\to \tE\GG$.

For any $\GG$-space $B$, the map $B^\GG_+\smsh\tE\GG\to B_+\smsh\tE\GG$ is an equivalence,
hence we get the two horizontal isomorphisms shown at the top and bottom of the diagram. This implies that the missing
top-right and bottom-right corners of the diagram are 0 groups, hence we get the vertical isomorphism
shown on the right.

Because we are assuming that $B^0\to B$ is a nonequivariant equivalence, we have that
$B^0 \times E\GG\to B\times E\GG$ is a $\GG$-equivalence. The inclusion $B^0\to B^\GG$ then
gives us a splitting of the map
\[
    \eta\colon H_\GG^R(B_+\smsh_B (E\GG)_+) \to H_\GG^R(B_+^\GG\smsh_B (E\GG)_+),
\]
showing that the third row of the diagram is a split short exact sequence.
\end{proof}

Let us assume further that $B$ is nonequivariantly simply connected and that all of the components
of $B^\GG$ are simply connected. 
As in Appendix~\ref{app:ordinarycohomology}, if those components are labeled $B^0$, \dots, $B^n$, then we know
that
\[
    RO(\Pi B) = \Z\{1,\sigma, \Omega_0,\ldots, \Omega_n\}/\rels{{\textstyle\sum \Omega_i = 2\sigma-2}}.
\]

\begin{definition}\label{def:even grading}
With $B$ as above,
let $\ROev(\Pi B) \subset RO(\Pi B)$ be the subgroup defined by
\[
    \ROev(\Pi B) = \left\{ a + b\sigma + {\textstyle\sum m_{i}\Omega_{i}} \mid \text{$a$ and $b$ even} \right\}.
\]
We call this the subgroup of \emph{even gradings.}
\end{definition}

\begin{corollary}\label{cor:even degrees}
Suppose that $B$ satisfies the condition of Theorem~\ref{thm:diagram} and that, in addition,
we know that the nonequivariant cohomology of $B^\GG$ is free abelian and concentrated
in even degrees. Then
\[
    \eta\colon H_\GG^{\ROev(\Pi B)}(B_+) \to H_\GG^{\ROev(\Pi B)}(B^\GG_+)
\]
is a monomorphism.
\end{corollary}

\begin{proof}
From Theorem~\ref{thm:diagram}, we know that we have an exact sequence
\[
    H_\GG^{RO(\Pi B)}(B/_B B^\GG\smsh_B (E\GG)_+) \to H_\GG^{RO(\Pi B)}(B_+) \to H_\GG^{RO(\Pi B)}(B^\GG_+)
\]
and that the first of these groups is the suspension of the cokernel of the split monomorphism
\[
    H_\GG^{RO(\Pi B)}(B_+\smsh_B (E\GG)_+) \to H_\GG^{RO(\Pi B)}(B^\GG_+\smsh_B (E\GG)_+)
\]
From our assumption on $B^\GG$ and the fact that the cohomology of $E\GG$ is concentrated in even degrees,
we get that $H_\GG^{RO(\Pi B)}(B^\GG_+\smsh_B (E\GG)_+)$ is concentrated in even degrees.
This tells us that $H_\GG^{RO(\Pi B)}(B/_B B^\GG\smsh_B (E\GG)_+)$, the suspension of the cokernel,
is 0 in even degrees. Hence we get that $\eta$ is a monomorphism in even degrees as claimed.
\end{proof}

This corollary is very useful for calculations, as it is easier to first calculate
the cohomology of $B^\GG$, which has trivial $\GG$-action.

We can go a step further: 
A diagram chase shows the following result, which we will not use in this paper,
but will need in \cite{CH:bt2}.

\begin{corollary}
Assume that $B$ satisfies the conditions in Corollary~\ref{cor:even degrees}.
Then the following is a pullback diagram:
\[
    \xymatrix{
        H_\GG^{\ROev(\Pi B)}(B_+) \ar[r]^\eta \ar[d]
            &  H_\GG^{\ROev(\Pi B)}(B_+^\GG) \ar[d] \\
        H_\GG^{\ROev(\Pi B)}(B_+\smsh_B (E\GG)_+) \ar@{>->}[r]^\eta
            &  H_\GG^{\ROev(\Pi B)}(B_+^\GG\smsh_B (E\GG)_+)
    }
\]
\qed
\end{corollary}

\section{The cohomology of the fixed set $B U(2)^\GG $}\label{sec:fixed sets}

From Corollary~\ref{cor:even degrees}, we know that
\[
    \eta\colon  H_\GG^\gr(B U(2)_+) \to  H_\GG^\gr(B U(2)^\GG _+)
\]
is injective in even degrees. We now compute the target of this map
and the images of the elements that we claim are generators of $H_\GG^\gr(B U(2)_+)$.

Write
\begin{align*}
 B U(2)^\GG  
 	&= \Gr_2(\C^\infty) \disjunion (\Gr_1(\C^\infty)\times\Gr_1({\Cq\infty})) \disjunion \Gr_2({\Cq\infty}) \\
    &= BU(2) \disjunion BT^2 \disjunion BU(2) \\
 	&= B^0 \disjunion B^1 \disjunion B^2.
\end{align*}
We have the following calculation.

\begin{proposition}\label{prop:fixedpoints}
The cohomology of $B U(2)^\GG $ is given by
% \begin{align*}
%   H_\GG^\gr(B U(2)^\GG _+) 
%   &\iso \HS[c_1, c_2, \zeta_1^{\pm 1}, \zeta_2^{\pm 1}] \\
%   &\quad\dirsum \HS[x_1, x_2, \zeta_0^{\pm 1}, \zeta_2^{\pm 1}] \\
%   &\quad\dirsum \HS[c_1, c_2, \zeta_0^{\pm 1}, \zeta_1^{\pm 1}].
% \end{align*}
\[
  H_\GG^\gr(B U(2)^\GG _+) 
  \iso \HS[c_1, c_2, \zeta_1^{\pm 1}, \zeta_2^{\pm 1}]
  \dirsum \HS[x_1, x_2, \zeta_0^{\pm 1}, \zeta_2^{\pm 1}]
  \dirsum \HS[c_1, c_2, \zeta_0^{\pm 1}, \zeta_1^{\pm 1}].
\]
Here, $\grad c_1 = 2$, $\grad c_2 = 4$, $\grad x_1 = 2$ and $\grad x_2 = 2$,
while $\grad \zeta_i = \Omega_{i}$.
\end{proposition}

\begin{proof}
This is proved similarly to \cite[Proposition~7.4]{Co:InfinitePublished}.
The element 
\[
 c_1 \in  H_\GG^2(\Gr_2(\C^\infty)_+) =  H_\GG^2(\Gr_2(\Cq\infty)_+)
\]
is the nonequivariant first Chern class, while
\[
 c_2 \in  H_\GG^4(\Gr_2(\C^\infty)_+) =  H_\GG^4(\Gr_2(\Cq\infty)_+)
\]
is the second Chern class.
The element 
\[
 x_1 \in  H_\GG^2((\Gr_1(\C^\infty)\times \Gr_1({\Cq\infty}))_+)
\]
is the 
first Chern class for the first
factor, while $x_2$ is the first Chern class for the second factor.
\end{proof}

Now we compute the images of various elements under $\eta$.
Recall that $\cw$ is the Euler class of $\omega$, the tautological 2-plane bundle over $BU(2)$,
$\cxw$ is the Euler class of $\chi\omega = \omega\tensor\Cq{}$, $\cl$ is the Euler class
of $\lambda = \ext^2\omega$, and $\cxl$ is the Euler class of $\chi\lambda$.
The elements $\zeta_0$, $\zeta_1$, and $\zeta_2$ come from the component structure of $BU(2)^\GG$
as in \cite[\S16]{Co:InfinitePublished}.

\begin{proposition}\label{prop:fixedformulas}
Write elements of the cohomology of $B U(2)^\GG $ as triples, per
Proposition~\ref{prop:fixedpoints}.
Under the restriction map 
\[
 \eta\colon  H_\GG^\gr(B U(2)_+) \to  H_\GG^\gr(B U(2)^\GG _+),
\]
we have
\begin{align*}
 \eta(\zeta_0) &= (\xi\zeta_1^{-1}\zeta_2^{-1},\ \zeta_0,\ \zeta_0) \\
 \eta(\zeta_1) &= (\zeta_1,\ \xi\zeta_0^{-1}\zeta_2^{-1},\ \zeta_1) \\
 \eta(\zeta_2) &= (\zeta_2,\ \zeta_2,\ \xi\zeta_0^{-1}\zeta_1^{-1}) \\
 \eta(\cl) &= (c_1\zeta_1,\ (e^2+\xi(x_1+x_2))\zeta_0^{-1}\zeta_2^{-1},\ c_1\zeta_1) \\
 \eta(\cxl ) &= ( (e^2+\xi c_1)\zeta_1^{-1},\ (x_1+x_2)\zeta_0\zeta_2,\ 
 							(e^2+\xi c_1)\zeta_1^{-1}) \\
 \eta(\cw ) &= (c_2\zeta_1\zeta_2^2,\ x_1(e^2 + \xi x_2)\zeta_0^{-1}\zeta_2,\ 
 							(e^4 + e^2\xi c_1 + \xi^2 c_2)\zeta_0^{-2}\zeta_1^{-1} ) \\
 \eta(\cxw) &= ( (e^4 + e^2\xi c_1 + \xi^2 c_2)\zeta_1^{-1}\zeta_2^{-2},\ 
 							x_2(e^2+\xi x_1)\zeta_0\zeta_2^{-1},\ c_2\zeta_0^2\zeta_1)
\end{align*}
\end{proposition}

\begin{proof}
The proofs are similar to the proofs of Propositions~7.5 and~8.6 of
\cite{Co:InfinitePublished}, though a new argument is needed in the cases of $\cw$ and $\cxw$.
We give the details for the calculation of $\eta(\cw) = (y_0,y_1,y_2)$.

We first note that
\[
    \grad\cw = \dim\omega = (4, 2+2\sigma, 4\sigma) = 4 + \Omega_1 + 2\Omega_2,
\]
which we determine by looking at the fibers of $\omega$, which are
$\C^2$ over $B^0$, $\Cpq 1{}$ over $B^1$, and $\Cq 2$ over $B^2$.

The first component of $\eta(\cw)$, $y_0$, lives in 
$H_\GG^{4+\Omega_1+2\Omega_2}(B^0_+)$.
By Proposition~\ref{prop:fixedpoints},
this group is isomorphic to $A(\GG)$, generated by $c_2\zeta_1\zeta_2^2$,
so $y_0 = \alpha c_2\zeta_1\zeta_2^2$
for some $\alpha\in A(\GG)$, say $\alpha = a + bg$.
We know that $y_0$ must restrict to $c_2$ nonequivariantly, hence $a + 2b = 1$.
On the other hand, $y_0^\GG$ is the Euler class of $\omega^\GG | B^0$, which is again
the tautological bundle over $B^0 = BU(2)$, hence $y_0^\GG = c_2$ as well, from which 
we see that $a = 1$. Thus, we must have $y_0 = c_2\zeta_1\zeta_2^2$.

The second component, $y_1$, lives in
\[
    H_\GG^{4+\Omega_1+2\Omega_2}(BT^2_+) = H_\GG^{2+2\sigma-\Omega_1+\Omega_2}(BT^2_+)
    \iso \Z^3
\]
with a basis given by $\{ e^2 x_1\zeta_0^{-1}\zeta_2, e^2 x_2\zeta_0^{-1}\zeta_2, \xi x_1x_2\zeta_0^{-1}\zeta_2 \}$.
Hence, we can write
\begin{equation}\label{eqn:w1}
    y_1 = (a_0e^2 x_1 + a_1e^2 x_2 + b\xi x_1x_2)\zeta_0^{-1}\zeta_2
\end{equation}
for some $a_0, a_1, b\in\Z$.
Reducing to nonequivariant cohomology, we get $bx_1x_2$, whereas $y_1$ should reduce to the
Euler class of the sum of the two tautological line bundles over $BT^2$, which is $x_1x_2$.
Hence $b = 1$.
Taking fixed points, $y_1^\GG$ is the Euler class of the first tautological line bundle over $BT^2$,
hence $y_1^\GG = x_1$. 
(Over $B^1$, the second tautological line bundle is taken with nontrivial action on its fibers,
hence it gives the 0 bundle on taking fixed points.)
But, taking the fixed points of (\ref{eqn:w1}) gives us $a_0x_1 + a_1x_2$,
hence $a_0 = 1$ and $a_1 = 0$. That is,
\[
    y_1 = (e^2 x_1 + \xi x_1x_2)\zeta_0^{-1}\zeta_2
\]
as claimed.

Finally, the component $y_2$ lives in
\[
    H_\GG^{4+\Omega_1+2\Omega_2}(B^2_+) = H_\GG^{4\sigma - 2\Omega_0 - \Omega_1}(BU(2)_+)
    \iso \Z^3
\]
with basis $\{e^4\zeta_0^{-2}\zeta_1^{-1}, e^2\xi c_1\zeta_0^{-2}\zeta_1^{-1}, \xi^2 c_2\zeta_0^{-2}\zeta_1^{-1}\}$.
Write
\begin{equation}\label{eqn:w2}
    y_2 = (a_1e^4 + a_2e^2\xi c_1 + a_3\xi^2 c_2)\zeta_0^{-2}\zeta_1^{-1}.
\end{equation}
In this case, reducing nonequivariantly and taking fixed points will allow us to see that
$a_1 = a_3 = 1$, but will not determine $a_2$. Instead, we consider the image of $y_2$
in the cohomology of $BT^2$ under the canonical map $s\colon BT^2\to BU(2)$.
There, $s^*y_2$ is the Euler class of $s^*\omega$, the sum of the two tautological line bundles over $BT^2$
with nontrivial action on their fibers.
From \cite{Co:InfinitePublished}, we know that
the Euler classes of these line bundles are (up to multiples of $\zeta$s) $e^2 + \xi x_1$
and $e^2 + \xi x_2$, and we know that the Euler class of a sum is the product of the Euler classes.
Hence, we must have
\begin{align*}
    s^* y_2 &= (e^2 + \xi x_1)(e^2 + \xi x_2)s^*(\zeta_0^{-2}\zeta_1^{-1}) \\
    &= (e^4 + e^2\xi(x_1 + x_2) + \xi^2 x_1x_2)s^*(\zeta_0^{-2}\zeta_1^{-1})
\end{align*}
Since $s^* c_1 = x_1 + x_2$ and $s^* c_2 = x_1x_2$, we see in (\ref{eqn:w2}) that
we must have $a_1 = a_2 = a_3 = 1$, verifying the claim in the proposition.
\end{proof}

\begin{corollary}\label{cor:relations}
In $H_\GG^\gr(B U(2)_+)$, we have the relations
\begin{align*}
	\zeta_0\zeta_1\zeta_2 &= \xi \\
	\zeta_1 \cxl &= (1-\kappa)\zeta_0\zeta_2 \cl + e^2 \\
	\zeta_2^2 \cxw &= (1-\kappa)\zeta_0^2 \cw + e^2 \cxl.
\end{align*}
\end{corollary}

\begin{proof}
Because both sides of these equalities live in even gradings, 
Corollary~\ref{cor:even degrees} tells us that it suffices to
verify that they are true after applying $\eta$.

For example, for the second relation, we have
\begin{align*}
    \eta(\zeta_1 \cxl)
    &= \eta(\zeta_1)\eta(\cxl) \\
    &= (\zeta_1,\ \xi\zeta_0^{-1}\zeta_2^{-1},\ \zeta_1) \\
    &\qquad   \cdot ( (e^2+\xi c_1)\zeta_1^{-1},\ (x_1+x_2)\zeta_0\zeta_2,\ 
 							(e^2+\xi c_1)\zeta_1^{-1}) \\
    &= ( e^2+\xi c_1,\ \xi(x_1+x_2),\ e^2+\xi c_1) \\
\intertext{while}
    \eta((1-\kappa)\zeta_0\zeta_2 \cl + e^2)
    &= (1-\kappa)(\xi c_1,\ e^2+\xi(x_1+x_2),\ \xi c_1) + (e^2,\ e^2,\ e^2) \\
    &= ( e^2+\xi c_1,\ \xi(x_1+x_2),\ e^2+\xi c_1)
\end{align*}
as well, using that $(1-\kappa)\xi = \xi$ and $(1-\kappa)e^2 = -e^2$. 
The other two relations are verified similarly.
\end{proof}

\section{Freeness of the proposed structure}\label{sec:freeness}

We know of the following elements in $ H_\GG^\gr(B U(2)_+)$, with their associated degrees:
\begin{align*}
	\zeta_0 && \grad\zeta_0 &= \Omega_0 \\
	\zeta_1 && \grad\zeta_1 &= \Omega_1 \\
	\zeta_2 && \grad\zeta_2 &= \Omega_2 \\
	\cw  &= e(\omega) & \grad \cw  &= 4 + \Omega_1 + 2\Omega_2 \\
	\cxw &= e(\chi\omega) & \grad \cxw &= 4 + 2\Omega_0 + \Omega_1 \\
	\cl &= e(\lambda) & \grad \cl &= 2 + \Omega_1 \\
	\cxl  &= e(\chi\lambda) & \grad \cxl  &= 2 + \Omega_0 + \Omega_2.
\end{align*}
As in the proof of Proposition~\ref{prop:fixedformulas},
the gradings of the Euler classes are the equivariant dimensions of
the bundles, which are
found by looking at the fibers of the bundles
over each fixed set component.
We will often use the names of the bundles for their dimensions, writing, for example,
\[
    \omega = 4 + \Omega_1 + 2\Omega_2.
\]

\begin{definition}\label{def:structure}
Let $ P^\gr$ be the $RO(\Pi BU(2))$-graded commutative ring defined as the algebra over
$\HS$ generated by elements $\zeta_0$, $\zeta_1$, $\zeta_2$, $\cw $,
$\cxw$, $\cl$, and $\cxl $, of the gradings above,
subject to the relations
\begin{align*}
	\zeta_0\zeta_1\zeta_2 &= \xi \\
	\zeta_1 \cxl  &= (1-\kappa)\zeta_0\zeta_2 \cl + e^2 \\
	\zeta_2^2 \cxw &= (1-\kappa)\zeta_0^2 \cw  + e^2 \cxl .
\end{align*}
\end{definition}

By Corollary~\ref{cor:relations}, we get a ring map
\[
    P^\gr \to H_\GG^\gr(B U(2)_+)
\]
which we will ultimately show is an isomorphism.
In the remainder of this section we prove the following key result.

\begin{proposition}\label{prop:freeness}
$ P^\gr$ is a free module over $\HS$.
It has a basis consisting of all monomials in the generators
$\zeta_0$, $\zeta_1$, $\zeta_2$, $\cw $,
$\cxw$, $\cl$, and $\cxl $
that are {\em not} multiples of
\begin{itemize}
\item $\zeta_0\zeta_1\zeta_2$
\item $\zeta_1 \cxl $
\item $\zeta_0^2\zeta_2^2 \cl$
\item $\zeta_2^2 \cxw$
\item $\zeta_0^3\zeta_1 \cw $
\item $\zeta_0^4 \cl\cw $.
\end{itemize}
\end{proposition}

\begin{proof}
We use Bergman's diamond lemma, Theorem~1.2 of \cite{Berg:diamondLemma}, as modified
for the commutative case by the comments in his \S10.3.
Let
\[
 E = \{ \zeta_0, \zeta_1, \zeta_2, \cl, \cxl , \cw , \cxw \}
\]
be the ordered set of generators and let $[E]$ be the free commutative monoid on $E$, i.e., the
set of monomials in the generators. Let $ Q$ be the ideal generated by the three relations
in the definition of $ P$, so that
\[
  P^\gr = \HS[E]/ Q.
\]
We impose the following order on $[E]$. We first give weights to the generators:
The weights of $\zeta_0$ and $\zeta_2$ are each 1; the weight of $\zeta_1$ is 2;
the weights of $\cl$ and $\cxl $ are 4 and 5, respectively;
and the weights of $\cw $ and $\cxw$ are 6 and 7, respectively.
This gives weights to each monomial in the obvious way. We then order the monomials first by weight,
then by lexicographical order for monomials of the same weight.
As needed to apply the diamond lemma, this order obeys the multiplicative property that,
if $A < B$, then $AC < BC$ for any monomial $C$.
It also obeys the descending chain condition, because there are only finitely many monomials
of any given weight.

We now specify a reduction system, consisting of pairs $(W,f)$ where $W\in [E]$ and
$f\in \HS[E]$ is a polynomial with each of its monomials preceding $W$ in the order on $[E]$.
We need further that $ Q$ is the ideal generated
by the collection of differences $W-f$ for all such pairs.
Here is the reduction system,\footnote{We thank the developers of Macaulay2, which was used to help find this reduction system in the form of a Gr\"obner basis.} 
with each pair $(W,f)$ written as $W\mapsto f$:
\begin{enumerate}[label=R\arabic*,series=reductions]
\item\label{red:1}
	$\zeta_0\zeta_1\zeta_2 \mapsto \xi$

\item\label{red:2}
	$\zeta_1 \cxl  \mapsto (1-\kappa)\zeta_0\zeta_2 \cl + e^2$

\item\label{red:3}
	$\zeta_0^2\zeta_2^2 \cl \mapsto \xi \cxl  + e^2\zeta_0\zeta_2$

\item\label{red:4}
	$\zeta_2^2 \cxw \mapsto (1-\kappa)\zeta_0^2 \cw  + e^2 \cxl $

\item\label{red:5}
	$\zeta_0^3\zeta_1 \cw  \mapsto \xi\zeta_2 \cxw - e^2\zeta_0^2\zeta_2 \cl + e^4\zeta_0$

\item\label{red:6}
	$\zeta_0^4 \cl\cw  \mapsto
		\xi \cxl \cxw + e^2\zeta_0^2 \cl \cxl  - e^2\zeta_0\zeta_2 \cxw$
\end{enumerate}
The reader can check that each monomial on the right precedes the monomial on the left in the order we defined on $[E]$.
Reductions \ref{red:1}, \ref{red:2}, and \ref{red:4} are the relations defining $ P$, so generate $ Q$.
The other three reductions follow from the relations so are also in $ Q$, as shown by the following
congruences modulo $ Q$:
\begin{align*}
	\zeta_0^2\zeta_2^2 \cl
		&= (1-\kappa)\zeta_0\zeta_2 \cdot (1-\kappa)\zeta_0\zeta_2 \cl \\
		&\equiv (1-\kappa)\zeta_0\zeta_1\zeta_2 \cxl  - (1-\kappa)e^2\zeta_0\zeta_2  \\
		&\equiv \xi \cxl  + e^2\zeta_0\zeta_2 \\
	\zeta_0^3\zeta_1 \cw 
		&= (1-\kappa)\zeta_0\zeta_1 \cdot (1-\kappa)\zeta_0^2 \cw  \\
		&\equiv (1-\kappa)\zeta_0\zeta_1\zeta_2^2 \cxw - (1-\kappa)e^2 \zeta_0\zeta_1 \cxl  \\
		&\equiv \xi\zeta_2 \cxw + e^2\zeta_0\zeta_1 \cxl  \\
		&\equiv \xi\zeta_2 \cxw - e^2\zeta_0^2\zeta_2 \cl + e^4\zeta_0 \\
	\zeta_0^4 \cl\cw 
		&= \zeta_0^2 \cl \cdot \zeta_0^2 \cw  \\
		&\equiv (1-\kappa)\zeta_0^2\zeta_2^2 \cl\cxw + e^2\zeta_0^2 \cl \cxl  \\
		&\equiv \xi \cxl \cxw + e^2\zeta_0^2 \cl \cxl  - e^2\zeta_0\zeta_2 \cxw
\end{align*}
The last step is to verify that we can resolve all ambiguities in this reduction system.
This somewhat tedious check we relegate to Appendix~\ref{app:resolution}.

Assuming the ambiguities resolved, the diamond lemma applies and proves the proposition.
\end{proof}

\section{Change of grading}\label{sec:change of grading}

In the following section we will be discussing maps between rings and modules with different gradings.
One example is the ring map $\rho\colon H_\GG^{RO(\GG)}(S^0) \to H^\Z(S^0;\Z)$ from the equivariant
cohomology of a point to the nonequivariant cohomology.
If $\alpha\in RO(\GG)$, this is a map
\[
    \rho\colon H_\GG^\alpha(S^0) \to H^{\rho(\alpha)}(S^0;\Z)
\]
where $\rho\colon RO(\GG)\to\Z$ is the augmentation map, $\rho(a+b\sigma) = a+b$.
It will be convenient in such cases to \emph{regrade} the target $H^\Z(S^0;\Z)$ on $RO(\GG)$
to get the ring $H^{RO(\GG)}(S^0;\Z)$ defined by
\[
    H^\alpha(S^0;\Z) := H^{\rho(\alpha)}(S^0;\Z).
\]

Now, the kernel of $\rho\colon RO(\GG)\to \Z$ is the subgroup generated by $\sigma - 1$.
In particular, we have
\[
    H^{\sigma-1}(S^0;\Z) = H^0(S^0;\Z)
\]
by definition. If we let $\iota\in H^{\sigma-1}(S^0;\Z)$ be the element corresponding to 1, then
multiplication by $\iota$ gives us the identification above, and in general multiplication by $\iota$
is an isomorphism $H^\alpha(S^0;\Z) \iso H^{\alpha+\sigma-1}(S^0;\Z)$, so $\iota$ is an invertible element
of the ring $H^{RO(\GG)}(S^0;\Z)$. In fact, we can write
\[
    H^{RO(\GG)}(S^0;\Z) = H^\Z(S^0;\Z)[\iota^{\pm 1}] = \Z[\iota^{\pm 1}]
\]
where $\grad \iota = \sigma - 1$.
(See also \cite[Proposition~3.8]{Co:InfinitePublished}.)

Another example is the map $\phi = (-)^\GG\colon H_\GG^{RO(\GG)}(S^0) \to H^\Z(S^0;\Z)$
that takes a cohomology class to its fixed point class \cite[Definition~1.13.26]{CostenobleWanerBook}.
In this case we have maps
\[
    \phi\colon H_\GG^\alpha(S^0) \to H^{\phi(\alpha)}(S^0;\Z)
\]
where $\phi\colon RO(\GG)\to \Z$ is given by $\phi(a+b\sigma) = (a+b\sigma)^\GG = a$.
The kernel of $\phi$ is generated by $\sigma$, so if we define
$H_\phi^{RO(\GG)}(S^0;\Z)$ by
\[
    H_\phi^\alpha(S^0;\Z) := H^{\alpha^\GG}(S^0;\Z),
\]
we get
\[
    H_\phi^{RO(\GG)}(S^0;\Z) = \Z[e^{\pm 1}]
\]
where $\grad e = \sigma$.

In general, if $A^R$ is a ring graded on an abelian group $R$ and $\gamma\colon Q\to R$ is a homomorphism,
we can regrade $A$ on $Q$ via
\[
    A^q := A^{\gamma(q)}
\]
to get the graded ring $A^Q$ (or $A_\gamma^Q$ if it is necessary to make unambiguous what homomorphism is being used). 
The following often holds in the cases we use.
(This result is related to \cite[Theorem~5.4]{Dade:GroupGraded} when $R$ is a trivial group.)

\begin{proposition}
If $A^R$ is a graded commutative ring and $\gamma\colon Q\to R$ is a split epimorphism of abelian groups, then
\[
    A^Q_\gamma \iso A^R[\zeta^\alpha \mid \alpha \in \ker\gamma]/
        \rels{\zeta^0 = 1,\ \zeta^\alpha\zeta^\beta = \zeta^{\alpha+\beta}}
\]
where $\grad\zeta^\alpha = \alpha$. (Note that the relations imply that each $\zeta^\alpha$ is invertible.)
\end{proposition}

\begin{proof}
For each $\alpha\in\ker\gamma$, we have $A^\alpha = A^0$. We let $\zeta^\alpha\in A^\alpha$
be the element corresponding to $1\in A^0$.
\end{proof}

In the cases we use, $\ker\gamma$ is in fact a free abelian group, in which case we need
adjoin just one invertible element for each generator of $\ker\gamma$.
Examples are the groups $H_\rho^{RO(\GG)}(S^0;\Z)$ and $H_\phi^{RO(\GG)}(S^0;\Z)$
above.

All of this generalizes to graded modules.
Suppose that $M^R$ is a module over $A^R$ with $\gamma\colon Q\to R$ 
a homomorphism.
We can regrade $M$ on $Q$ to get a module $M^Q$ over $A^Q$.
If $M^R$ is free over $A^R$, then $M^Q$ is a free module over $A^Q$.
Moreover, modulo multiplication by invertible elements, there is a correspondence between
bases for $M^Q$ over $A^Q$ and bases for $M^R$ over $A^R$.

\section{The cohomology of $B U(2)$}\label{sec:mainresult}

As noted when we defined $P^\gr$, we have a ring map $P^\gr\to H_\GG^\gr(B U(2)_+)$.
There are two possible methods of proof to show that this map is an isomorphism.
One is the line of argument used in \cite{Co:InfinitePublished}, for which we would need
to show two isomorphisms,
\[
     P^\gr\tensorS H_\GG^{RO(\GG)}((E\GG)_+) \iso
         H_\GG^\gr(B U(2)_+\smsh (E\GG)_+)
\]
and
\[
     P^\gr\tensorS H_\GG^{RO(\GG)}(\tE\GG) \iso
         H_\GG^\gr(B U(2)_+\smsh \tE\GG).
\]
The simpler alternative we will use here is based on the result behind the proof used in \cite{CHTFiniteProjSpace},
which is the following.

\begin{proposition}[\cite{CHTFiniteProjSpace}, Proposition~4.4]\label{prop:proofofbasis}
Let $X$ be a $C_2$-ex-space over $Y$ of finite type and let
\[
 \sX := \big\{ x_\alpha \in H_{C_2}^{\gamma_\alpha}(X) \big\}
\]
be a collection of cohomology elements.
If $ \rho^*\sX$ is a basis 
for $H_\rho^{RO(\Pi Y)}(X;\Z)$
and $\sX^{C_2}$ is a basis for $H_\phi^{RO(\Pi Y)}(X^{C_2};\Z)$
(both as $H^{RO(\GG)}(S^0;\Z)$-modules),
then $\sX$ is a basis for $H_{C_2}^{RO(\Pi Y)}(X)$ as a module over
$\HS$.
\qed
\end{proposition}

We will apply this result with $X = Y = B U(2)$ and
$\sX$ being the image of a basis for $ P^\gr$.
In order to use this result, we need to examine the rings
$H_\rho^\gr(BU(2)_+;\Z)$ and
$H_\phi^\gr((B U(2))^\GG_+;\Z)$ and the restriction maps to these from
the equivariant cohomology of $B U(2)$.

\subsection{Restriction to nonequivariant cohomology}
We start with the map
\[
    \rho\colon H_\GG^\gr(B U(2)_+) \to H_\rho^\gr(BU(2)_+;\Z).
\]
We are (re)grading the ring on the right on $RO(\Pi BU(2))$
via the restriction map $\rho\colon RO(\Pi BU(2))\to \Z$ given by
\[
    \rho(1) = 1 \qquad \rho(\sigma) = 1 \qquad \rho(\Omega_i) = 0.
\]
The kernel of this map is the subgroup generated by the $\Omega_i$ together with $\sigma-1$,
and there are corresponding
units in $H^\gr(BU(2)_+;\Z)$ which we can identify with the elements $\rho(\zeta_i)$
and the element we have called $\iota$, respectively.
We write $\zeta_i = \rho(\zeta_i)$ again.
Note that
\[
    \zeta_0\zeta_1\zeta_2 = \iota^2.
\]

As in the equivariant case, write $\omega$ for the tautological line bundle
over $BU(2)$, and let $c_1 = c_1(\omega)$ and $c_2 = c_2(\omega)$ be its
Chern classes.
Combining the familiar structure of the cohomology of $BU(2)$ with the invertible elements introduced
by the regrading on $RO(\Pi BU(2))$, we get the following calculation.
\begin{equation}\label{eqn:nonequivariant}
    H^\gr(BU(2)_+;\Z) = \Z[c_1,c_2,\iota^{\pm 1},\zeta_i^{\pm 1}]
    \big/\rels{\zeta_0\zeta_1\zeta_2 = \iota^2}
\end{equation}
where
\[
    \grad c_1 = 2,\ \grad c_2 = 4,\ \grad\iota = \sigma - 1, \text{ and } \grad\zeta_i = \Omega_i.
\]

The restrictions of various elements from the equivariant
cohomology of $B U(2)$ are the following.
\begin{align*}
    \rho(\cw) &= \zeta_1\zeta_2^2 c_2 & \rho(\cxw) &= \zeta_0^2\zeta_1 c_1 \\
    \rho(\cl) &= \zeta_1 c_1 & \rho(\cxl) &= \zeta_0\zeta_2 c_1 \\
\end{align*}

The main calculation we need is the following.

\begin{proposition}\label{prop:rho iso}
\[
    P^\gr\tensorS H_\rho^{RO(\GG)}(S^0;\Z) \iso 
    \Z[\cw,\cl,\iota^{\pm 1},\zeta_i^{\pm 1}]
    \big/\rels{\zeta_0\zeta_1\zeta_2 = \iota^2}
\]
\end{proposition}

\begin{proof}
We examine what happens with the relations defining $P^{RO(\GG)}$ when we change
the ground ring from $H_\GG^{RO(\GG)}(S^0)$ to $H^{RO(\GG)}(S^0;\Z)$.
\begin{align*}
    \zeta_0\zeta_1\zeta_2 &= \xi & \text{becomes } & & \zeta_0\zeta_1\zeta_2 &= \iota^2 \\
    \zeta_1 \cxl &= (1-\kappa)\zeta_0\zeta_2 \cl + e^2 & \text{becomes }
        & & \zeta_1 \cxl &= \zeta_0\zeta_2 \cl \\
    \zeta_2^2 \cxw &= (1-\kappa)\zeta_0^2 \cw + e^2 \cxl & \text{becomes }
        & & \zeta_2^2 \cxw &= \zeta_0^2 \cw
\end{align*}
Given the invertibility of the $\zeta_i$, these relations now simply express
$\cxl$, and $\cxw$ in terms of $\cl$ and $\cw$.
The result is then clear.
\end{proof}

\begin{corollary}\label{cor:rho}
The composite
\[
    P^\gr\to H_\GG^\gr(B U(2)_+) \xrightarrow{\rho}
        H_\rho^\gr(BU(2)_+;\Z)
\]
takes a basis for $P^\gr$ over $\HS$
to a basis for $H_\rho^\gr(BU(2)_+;\Z)$ over $H_\rho^{RO(\GG)}(S^0)$.
\end{corollary}

\begin{proof}
From the preceding proposition, the calculation (\ref{eqn:nonequivariant}),
the fact that $\rho(\cl) = \zeta_1 c_1$ and $\rho(\cw) = \zeta_1\zeta_2^2 c_2$,
and the invertibility of the $\zeta_i$, we get an isomorphism
\[
    P^\gr\tensorS H_\rho^{RO(\GG)}(S^0) \iso H_\rho^\gr(BU(2)_+;\Z).
\]
The statement in the corollary now follows from the freeness of $P^\gr$.
\end{proof}

\subsection{Restriction to fixed points}\label{subsec:restrict fixed}
We now consider the map
\[
    \phi = (-)^\GG\colon H_\GG^\gr(B U(2)_+)
    \to H_\phi^\gr(B U(2)^\GG_+;\Z)
\]
We expand a bit more on the grading on the right.
Recall that we have
\[
    B U(2)^\GG = B^0 \disjunion B^1 \disjunion B^2
\]
with 
\[
    B^0 \homeo B^2 \homeo BU(2) \quad\text{and}\quad B^1 \homeo BT^2.
\]
We first think of $H^\gr(B U(2)^\GG_+;\Z)$ as
graded on $\Z^3$ via
\[
    H^a(B U(2)^\GG_+;\Z) = H^{a_0}(B^0_+;\Z) \dirsum H^{a_1}(B^1_+;\Z) \dirsum H^{a_2}(B^2_+;\Z)
\]
if $a = (a_0,a_1,a_2) \in \Z^3$. We then have the fixed-point map
\[
    \phi\colon RO(\Pi BU(2)) \to \Z^3
\]
given by
\[
    \textstyle \phi(a + b\sigma + \sum_i m_i\Omega_i)
    = (a - 2m_0, a-2m_1, a-2m_2).
\]
Thus, if $\alpha = a + b\sigma + \sum_i m_i\Omega_i$, then
\[
    H^\alpha(B U(2)^\GG_+;\Z) = 
    H^{a-2m_0}(B^0_+;\Z) \dirsum H^{a-2m_1}(B^1_+;\Z) \dirsum H^{a-2m_2}(B^2_+;\Z).
\]

If we now examine the $i$th summand, we are regrading via a homomorphism
$RO(\Pi BU(2))\to \Z$ whose kernel has a basis given by $\sigma$ and
the $\Omega_{j}$ with $j\neq i$.
There are corresponding invertible elements given by
$e = e^\GG$ and $\zeta_j = \zeta_j^\GG$.
This gives us the calculation
\begin{multline}\label{eqn:fixedpoints}
    H^\gr(B U(2)^\GG)_+;\Z) \iso \\
    \Z[c_1,c_2,e^{\pm 1},\zeta_1^{\pm 1},\zeta_2^{\pm 1}]
    \dirsum \Z[x_1,x_2, e^{\pm 1}, \zeta_0^{\pm 1}, \zeta_2^{\pm 1}]
    \dirsum \Z[c_1,c_2,e^{\pm 1},\zeta_0^{\pm 1},\zeta_1^{\pm 1}]
\end{multline}
where
\begin{align*}
    \grad c_1 &= 2,\ \grad c_2 = 4,\ \grad x_1 = \grad x_2 = 2, \\
    \grad e &= \sigma, \text{ and } \grad\zeta_j = \Omega_j.
\end{align*}

Recall that $H_\phi^{RO(\GG)}(S^0;\Z) \iso \Z[e^{\pm 1}]$.
With this in mind, we can rewrite (\ref{eqn:fixedpoints}) as
\begin{multline}\label{eqn:fixedpoints2}
    H^\gr(B U(2)^\GG)_+;\Z) \iso \\
    H_\phi^{RO(\GG)}(S^0;\Z)[c_1,c_2,\zeta_1^{\pm 1},\zeta_2^{\pm 1}]
    \dirsum H_\phi^{RO(\GG)}(S^0;\Z)[x_1,x_2, \zeta_0^{\pm 1}, \zeta_2^{\pm 1}] \\
    \dirsum H_\phi^{RO(\GG)}(S^0;\Z)[c_1,c_2,\zeta_0^{\pm 1},\zeta_1^{\pm 1}]
\end{multline}

\begin{proposition}\label{prop:fixed point iso}
The map
\[
    P^\gr\tensorS H_\phi^{RO(\GG)}(S^0;\Z) \xrightarrow{\bar\phi} 
    H_\phi^\gr(B U(2)^\GG)_+;\Z)
\]
induced by $\phi$ is an isomorphism of $H_\phi^{RO(\GG)}(S^0;\Z)$-modules.
\end{proposition}

\begin{proof}
We first record the value of $\bar\phi$ on various elements, where we write elements in the codomain
as triples, using (\ref{eqn:fixedpoints2}).
These values can be read off from the values of $\eta$ given in Proposition~\ref{prop:fixedformulas}.
\begin{align*}
    \bar\phi(\zeta_0) &= (0, \zeta_0, \zeta_0) \\
    \bar\phi(\zeta_1) &= (\zeta_1, 0, \zeta_1) \\
    \bar\phi(\zeta_2) &= (\zeta_2, \zeta_2, 0) \\
    \bar\phi(\cl) &= ( c_1\zeta_1, e^2\zeta_0^{-1}\zeta_2^{-1}, c_1\zeta_1 ) \\
    \bar\phi(\cxl) &= ( e^2\zeta_1^{-1}, (x_1+x_2)\zeta_0\zeta_2, e^2\zeta_1^{-1} ) \\
    \bar\phi(\cw) &= ( c_2\zeta_1\zeta_2^2, e^2x_1\zeta_0^{-1}\zeta_2, e^4\zeta_0^{-2}\zeta_1^{-1} ) \\
    \bar\phi(\cxw) &= ( e^4\zeta_1^{-1}\zeta_2^{-2}, e^2x_2\zeta_0\zeta_2^{-1}, c_2\zeta_0^2\zeta_1 )
\end{align*}
From these we get the following calculations.
\begin{align*}
    \bar\phi(e^{-2}\cl\cxl) &= (c_1, x_1+x_2, c_1) \\
    \bar\phi(e^{-4}\cw\cxw) &= (c_2, x_1x_2, c_2) \\
    \bar\phi( e^{-6}\zeta_1^2\zeta_2^2\cxl\cxw) &= (1,0,0) \\
    \bar\phi( e^{-2}\zeta_0\zeta_2\cl) &= (0,1,0) \\
    \bar\phi( e^{-6}\zeta_0^2\zeta_1^2 \cxl\cw) &= (0,0,1) \\
    \bar\phi(e^{-6}\zeta_1\zeta_2^2\cxl\cxw) &= (\zeta_1^{-1},0,0) \\
    \bar\phi(e^{-6}\zeta_1^2\zeta_2\cxl\cxw) &= (\zeta_2^{-1},0,0) \\
    \bar\phi(e^{-2}\zeta_2\cl(1 - e^{-2}\zeta_1\cxl)) &= (0,\zeta_0^{-1},0) \\
    \bar\phi(e^{-2}\zeta_0\cl(1 - e^{-2}\zeta_1\cxl)) &= (0,\zeta_2^{-1},0) \\
    \bar\phi(e^{-6}\zeta_0\zeta_1^2 \cxl\cw) &= (0,0,\zeta_0^{-1}) \\
    \bar\phi(e^{-6}\zeta_0^2\zeta_1 \cxl\cw) &= (0,0,\zeta_1^{-1}) \\
\end{align*}
We can now define an inverse $\psi$ to $\bar\phi$ as the unique algebra map with the following values:
\begin{align*}
    \psi(1,0,0) &= e^{-6}\zeta_1^2\zeta_2^2\cxl\cxw \\
    \psi(0,1,0) &= e^{-2}\zeta_0\zeta_2\cl \\
    \psi(0,0,1) &= e^{-6}\zeta_0^2\zeta_1^2 \cxl\cw \\
    \psi(\zeta_j,0,0) &= \zeta_j\psi(1,0,0) \text{ etc.} \\
    \psi(\zeta_1^{-1},0,0) &= e^{-6}\zeta_1\zeta_2^2\cxl\cxw \text{ etc.} \\
    \psi(c_1,0,0) &= e^{-2}\cl\cxl\psi(1,0,0) \\
    \psi(c_2,0,0) &= e^{-4}\cw\cxw\psi(1,0,0) \\
    \psi(0,x_1,0) &= e^{-2}\zeta_0\cw\psi(0,\zeta_2^{-1},0) \\
    \psi(0,x_2,0) &= e^{-2}\zeta_2\cxw\psi(0,\zeta_0^{-1},0) \\
    \psi(0,0,c_1) &= e^{-2}\cl\cxl\psi(0,0,1) \\
    \psi(0,0,c_2) &= e^{-4}\cw\cxw\psi(0,0,1)
\end{align*}
We can then check that $\psi$ is well-defined and is the inverse of $\bar\phi$, hence $\bar\phi$ is an isomorphism.
(The checks are tedious but straightforward using the relations in $P^\gr\tensorS H_\phi^{RO(\GG)}(S^0;\Z)$.)
\end{proof}

\begin{corollary}\label{cor:fixedpoints}
The composite
\[
    P^\gr \to H_\GG^\gr(B U(2)_+)
    \xrightarrow{\phi} H_\phi^\gr(B U(2)^\GG_+;\Z)
\]
takes a basis for $P^\gr$ over $\HS$ to a basis for
$H_\phi^\gr(B U(2)^\GG_+;\Z)$ over $H_\phi^{RO(\GG)}(S^0;\Z)$.
\qed
\end{corollary}

\subsection{The cohomology of $B U(2)$}
We can now prove Theorem~\ref{thm:A}, which we restate here with slightly more information.

\begin{theorem}\label{thm:main}
$ H_\GG^\gr(B U(2)_+)$ is generated as a commutative algebra over $\HS$
by the elements $\zeta_0$, $\zeta_1$, $\zeta_2$ $\cl$, $\cxl$, $\cw$, and $\cxw$, modulo the relations
\begin{align*}
	\zeta_0\zeta_1\zeta_2 &= \xi \\
	\zeta_1 \cxl  &= (1-\kappa)\zeta_0\zeta_2 \cl + e^2 \\
	\zeta_2^2 \cxw &= (1-\kappa)\zeta_0^2 \cw  + e^2 \cxl .
\end{align*}
It is a free module over $\HS$, with the basis given in
Proposition~\ref{prop:freeness}.
\end{theorem}

\begin{proof}
Proposition~\ref{prop:proofofbasis} and
Corollaries~\ref{cor:rho} and~\ref{cor:fixedpoints}
imply that a basis of $ P^\gr$ over $\HS$ is taken to a basis
of $ H_\GG^\gr(B U(2)_+)$,
hence the ring map $ P^\gr\to  H_\GG^\gr(B U(2)_+)$
is an isomorphism.
The theorem then follows from the definition of $ P^\gr$
and Proposition~\ref{prop:freeness}.
\end{proof}

\section{Examples of bases}\label{sec:bases}

It may be illuminating to see the locations of some of the basis elements.
The specific basis we use is the one given by Proposition~\ref{prop:freeness},
but the locations of the elements would not change if we used a different basis.
We look at ``$RO(\GG)$ pages,'' that is, 
at the groups graded by cosets $\alpha+RO(\GG)$
for fixed $\alpha\in RO(\Pi BU(2))$.

As the first example, we consider the $RO(\GG)$ grading itself, which we think of as the coset
of gradings $0 + RO(\GG)$. 
Below, we will also list the fixed sets of the basic elements, using the map
\[
    (-)^\GG\colon H_\GG^\gr(B U(2)_+) \to H^{\Z^3}(B U(2)^\GG_+;\Z),
\]
grading the target on $\Z^3$ rather than regrading on $RO(\Pi BU(2))$.
(See \S\ref{subsec:restrict fixed} for more details.)

The beginning of the list of basic elements in the $RO(\GG)$ grading
is as follows, along with the grading of each and its fixed sets.
\begin{align*}
    &x && \grad x && x^\GG \\
    &1 && 0 && (1,1,1) \\
    &\zeta_0\zeta_2\cl && 2\sigma && (0,1,0) \\
    &\zeta_0^2\zeta_1\cw && 4\sigma && (0,0,1) \\
    &\cl\cxl && 2 + 2\sigma && (c_1,x_1+x_2,c_1) \\
    &\zeta_0\zeta_2\cl^2\cxl && 2 + 4\sigma && (0,x_1+x_2,0) \\
    &\zeta_0^2\cl\cw && 2 + 4\sigma && (0,x_1,c_1) \\
    &\zeta_0^3\zeta_2\cl^2\cw && 2 + 6\sigma && (0,x_1,0) \\
    &\cl^2\cxl^2 && 4 + 4\sigma && (c_1^2,(x_1+x_2)^2,c_1^2) \\
    &\cw\cxw && 4 + 4\sigma && (c_2,x_1x_2,c_2) \\
    &\zeta_0\zeta_2\cl^3\cxl^2 && 4 + 6\sigma && (0,(x_1+x_2)^2,0) \\
    &\zeta_0\zeta_2\cl\cw\cxw && 4 + 6\sigma && (0,x_1x_2,0) \\
    &\zeta_0^2\cl^2\cxl\cw && 4 + 6\sigma && (0,x_1(x_1+x_2),c_1^2) \\
    &\zeta_0^3\zeta_2\cl^3\cxl\cw && 4 + 8\sigma && (0,x_1(x_1+x_2),0) \\
    &\zeta_0^2\zeta_1\cw^2\cxw && 4 + 8\sigma && (0,0,c_2)
\end{align*}
We draw the locations of these basis elements (and more) on a grid in which the location $a + b\sigma\in RO(\GG)$
is shown at point $(a,b)$. Because $a$ and $b$ are always even, the spacing of the grid lines is every 2, not 1.
The numbers in each circle represent the number of basis elements in that grading.
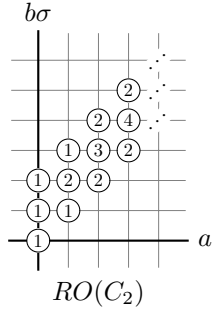
\begin{center}
\begin{tikzpicture}[scale=0.4] 
    \draw[step=1cm,gray,very thin] (-0.9,-0.9) grid (4.9,6.9);
    \draw[thick] (-1,0) -- (5,0);
    \draw[thick] (0,-1) -- (0,7);
    \node[right] at (5,0) {$a$};
    \node[above] at (0,7) {$b\sigma$};
    \node[below] at (2,-1) {$RO(\GG)$};

    \node[fill=white,scale=0.7] at (0,0) {1}; \draw (0,0) circle(0.38cm);
    \node[fill=white,scale=0.7] at (0,1) {1}; \draw (0,1) circle(0.38cm);
    \node[fill=white,scale=0.7] at (0,2) {1}; \draw (0,2) circle(0.38cm);

    \node[fill=white,scale=0.7] at (1,1) {1}; \draw (1,1) circle(0.38cm);
    \node[fill=white,scale=0.7] at (1,2) {2}; \draw (1,2) circle(0.38cm);
    \node[fill=white,scale=0.7] at (1,3) {1}; \draw (1,3) circle(0.38cm);

    \node[fill=white,scale=0.7] at (2,2) {2}; \draw (2,2) circle(0.38cm);
    \node[fill=white,scale=0.7] at (2,3) {3}; \draw (2,3) circle(0.38cm);
    \node[fill=white,scale=0.7] at (2,4) {2}; \draw (2,4) circle(0.38cm);

    \node[fill=white,scale=0.7] at (3,3) {2}; \draw (3,3) circle(0.38cm);
    \node[fill=white,scale=0.7] at (3,4) {4}; \draw (3,4) circle(0.38cm);
    \node[fill=white,scale=0.7] at (3,5) {2}; \draw (3,5) circle(0.38cm);

    \node[fill=white,scale=0.8,rotate=45] at (4,4) {\dots};
    \node[fill=white,scale=0.8,rotate=45] at (4,5) {\dots};
    \node[fill=white,scale=0.8,rotate=45] at (4,6) {\dots};

\end{tikzpicture}
\end{center}
Notice that, if we reduce each basis element using the non-regraded map
\[
    \rho\colon H_\GG^\gr(B U(2)_+) \to H^\Z(BU(2)_+;\Z),
\]
we get the familiar basis (shown in the same order as the elements above)
\[
    \{ 1, c_1, c_2, c_1^2, c_1^3, c_1c_2, c_1^2c_2, c_1^4, c_2^2, c_1^5, c_1c_2^2, c_1^3c_2, c_1^4c_2, c_2^3, \dots \}.
\]
The equivariant basis elements
that appear on the diagonal line of gradings $a+b\sigma$ with $a + b = n$
are the ones that reduce to nonequivariant basis elements
with total degree $n$.
We can illustrate the action of $\rho$ as follows, with the lower horizontal
line representing the nonequivariant cohomology of $BU(2)$.
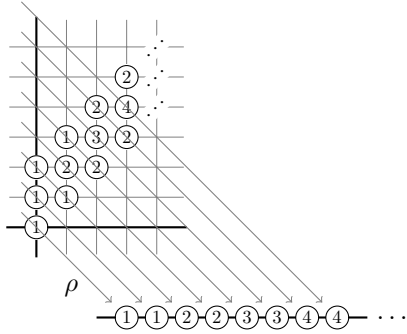
\begin{center}
\begin{tikzpicture}[scale=0.4] 

    \draw[step=1cm,gray,very thin] (-0.9,-0.9) grid (4.9,6.9);
    \draw[thick] (-1,0) -- (5,0);
    \draw[thick] (0,-1) -- (0,7);

    \node[fill=white,scale=0.7] at (0,0) {1}; \draw (0,0) circle(0.38cm);
    \node[fill=white,scale=0.7] at (0,1) {1}; \draw (0,1) circle(0.38cm);
    \node[fill=white,scale=0.7] at (0,2) {1}; \draw (0,2) circle(0.38cm);

    \node[fill=white,scale=0.7] at (1,1) {1}; \draw (1,1) circle(0.38cm);
    \node[fill=white,scale=0.7] at (1,2) {2}; \draw (1,2) circle(0.38cm);
    \node[fill=white,scale=0.7] at (1,3) {1}; \draw (1,3) circle(0.38cm);

    \node[fill=white,scale=0.7] at (2,2) {2}; \draw (2,2) circle(0.38cm);
    \node[fill=white,scale=0.7] at (2,3) {3}; \draw (2,3) circle(0.38cm);
    \node[fill=white,scale=0.7] at (2,4) {2}; \draw (2,4) circle(0.38cm);

    \node[fill=white,scale=0.7] at (3,3) {2}; \draw (3,3) circle(0.38cm);
    \node[fill=white,scale=0.7] at (3,4) {4}; \draw (3,4) circle(0.38cm);
    \node[fill=white,scale=0.7] at (3,5) {2}; \draw (3,5) circle(0.38cm);

    \node[fill=white,scale=0.8,rotate=45] at (4,4) {\dots};
    \node[fill=white,scale=0.8,rotate=45] at (4,5) {\dots};
    \node[fill=white,scale=0.8,rotate=45] at (4,6) {\dots};

    \draw[thick] (2,-3) -- (11,-3);
%    \draw[thick] (3,-2.1) -- (3,-3.9);
    \node[fill=white,scale=0.7] at (3,-3) {1}; \draw (3,-3) circle(0.38cm);
    \node[fill=white,scale=0.7] at (4,-3) {1}; \draw (4,-3) circle(0.38cm);
    \node[fill=white,scale=0.7] at (5,-3) {2}; \draw (5,-3) circle(0.38cm);
    \node[fill=white,scale=0.7] at (6,-3) {2}; \draw (6,-3) circle(0.38cm);
    \node[fill=white,scale=0.7] at (7,-3) {3}; \draw (7,-3) circle(0.38cm);
    \node[fill=white,scale=0.7] at (8,-3) {3}; \draw (8,-3) circle(0.38cm);
    \node[fill=white,scale=0.7] at (9,-3) {4}; \draw (9,-3) circle(0.38cm);
    \node[fill=white,scale=0.7] at (10,-3) {4}; \draw (10,-3) circle(0.38cm);

    \draw[->,gray] (-0.5,0.5) -- (2.5,-2.5);
    \draw[->,gray] (-0.5,1.5) -- (3.5,-2.5);
    \draw[->,gray] (-0.5,2.5) -- (4.5,-2.5);
    \draw[->,gray] (-0.5,3.5) -- (5.5,-2.5);
    \draw[->,gray] (-0.5,4.5) -- (6.5,-2.5);
    \draw[->,gray] (-0.5,5.5) -- (7.5,-2.5);
    \draw[->,gray] (-0.5,6.5) -- (8.5,-2.5);
    \draw[->,gray] (-0.5,7.5) -- (9.5,-2.5);
    \node[left] at (1.7,-2) {$\rho$};
    \node[right] at (11,-3) {\dots};

\end{tikzpicture}
\end{center}

The fixed sets listed above illustrate another feature of this basis. The basis elements in any vertical line
restrict to a basis for the fixed sets in a given grading. For example, the four basis elements in gradings of the form
$2 + b\sigma$ have fixed sets that give a basis for
\[
    H^2((B U(2)^\GG_+;\Z) \iso H^2(BU(2)_+;\Z)\dirsum H^2(BT^2_+;\Z)\dirsum H^2(BU(2)_+;\Z).
\]
A more familiar basis for this group would be
\[
    \{ (c_1,0,0), (0,x_1,0), (0,x_2,0), (0,0,c_1) \},
\]
but it's not hard to see that the fixed sets shown in the list above form another basis.

The fact that the basis elements restrict to a nonequivariant basis and to a basis of the fixed sets
is a reflection and illustration of Proposition~\ref{prop:proofofbasis}.

\begin{remark}
One last comment about the $RO(\GG)$-graded page. This is the more familiar $RO(\GG)$-graded cohomology
and forms a subring of the full $RO(\Pi BU(2))$-graded ring.
To describe the multiplicative structure of this subring, in terms of generators and relations,
is doable but complicated.
We get a much simpler description after enlarging the grading as we have done.
See, for example, \cite[Proposition~1.7]{DuggerGrass}, where Dugger calculates the $RO(\GG)$-graded cohomology
of the real Grassmannian with constant $\Z/2$ coefficients, an analogous situation.
He comments that the ``complexity of the \dots\ description is discouraging,''
which we believe is because he is looking at only the $RO(\GG)$-graded part of a larger algebra
whose description would be simpler (and likely very similar to the complex case we have calculated here).
\end{remark}

We give another example, both to show that the arrangement of basis elements varies depending on the page,
and for use in \S\ref{sec:char numbers}.
We look at the coset $\omega + RO(\GG) = \Omega_1 + 2\Omega_2 + RO(\GG)$.
Here is the beginning of the list of basis elements in this page:
\begin{align*}
    &x && \grad x = \Omega_1+2\Omega_2+ {} && x^\GG \\
    &\zeta_1\zeta_2^2 && 0 && (1,0,0) \\
    &\zeta_2^2\cl && 2 && (c_1,1,0) \\
    &\zeta_0\zeta_2^3\cl^2 && 2+2\sigma && (0,1,0) \\
    &\cw && 4 && (c_2,x_1,1) \\
    &\zeta_2^2\cl^2\cxl && 4+2\sigma && (c_1^2,x_1,0) \\
    &\zeta_0\zeta_2 \cl\cw && 4 + 2\sigma && (0,x_1,0) \\
    &\zeta_0^2\zeta_1\cw^2 && 4 + 4\sigma && (0,0,1) \\
    &\zeta_0\zeta_2^3\cl^3\cxl && 4 + 4\sigma && (0, x_1+x_2, 0)
\end{align*}
Here is the arrangement of the basis elements:
% \begin{center}
% \begin{tikzpicture}[scale=0.4] 
%     \draw[step=1cm,gray,very thin] (-0.9,-0.9) grid (3.9,3.9);
%     \draw[thick] (-1,0) -- (4,0);
%     \draw[thick] (0,-1) -- (0,4);
%     \node[right] at (4,0) {$a$};
%     \node[above] at (0,4) {$b\sigma$};
%     \node[below] at (1.5,-1) {$\Omega_1+2\Omega_2 + RO(\GG)$};

%     \draw[fill] (0,0) circle(0.2cm);

%     \draw[fill] (1,0) circle(0.2cm);
%     \draw[fill] (1,1) circle(0.2cm);

%     \draw[fill] (2,0) circle(0.2cm);
%     \draw[fill] (2,1) circle(0.2cm); \draw (2,1) circle(0.3cm);
%     \draw[fill] (2,2) circle(0.2cm); \draw (2,2) circle(0.3cm);

%     \draw[fill] (2.9,0.9) circle(0.05cm); \draw[fill] (3.1,1.1) circle(0.05cm); \draw[fill] (3.3,1.3) circle(0.05cm);
%     \draw[fill] (2.9,1.9) circle(0.05cm); \draw[fill] (3.1,2.1) circle(0.05cm); \draw[fill] (3.3,2.3) circle(0.05cm);
%     \draw[fill] (2.9,2.9) circle(0.05cm); \draw[fill] (3.1,3.1) circle(0.05cm); \draw[fill] (3.3,3.3) circle(0.05cm);

% \end{tikzpicture}
% \end{center}
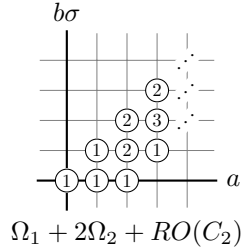
\begin{center}
\begin{tikzpicture}[scale=0.4] 
    \draw[step=1cm,gray,very thin] (-0.9,-0.9) grid (4.9,4.9);
    \draw[thick] (-1,0) -- (5,0);
    \draw[thick] (0,-1) -- (0,5);
    \node[right] at (5,0) {$a$};
    \node[above] at (0,5) {$b\sigma$};
    \node[below] at (2,-1) {$\Omega_1+2\Omega_2 + RO(\GG)$};

    \node[fill=white,scale=0.7] at (0,0) {1}; \draw (0,0) circle(0.38cm);

    \node[fill=white,scale=0.7] at (1,0) {1}; \draw (1,0) circle(0.38cm);
    \node[fill=white,scale=0.7] at (1,1) {1}; \draw (1,1) circle(0.38cm);

    \node[fill=white,scale=0.7] at (2,0) {1}; \draw (2,0) circle(0.38cm);
    \node[fill=white,scale=0.7] at (2,1) {2}; \draw (2,1) circle(0.38cm);
    \node[fill=white,scale=0.7] at (2,2) {2}; \draw (2,2) circle(0.38cm);

    \node[fill=white,scale=0.7] at (3,1) {1}; \draw (3,1) circle(0.38cm);
    \node[fill=white,scale=0.7] at (3,2) {3}; \draw (3,2) circle(0.38cm);
    \node[fill=white,scale=0.7] at (3,3) {2}; \draw (3,3) circle(0.38cm);

    \node[fill=white,scale=0.8,rotate=45] at (4,2) {\dots};
    \node[fill=white,scale=0.8,rotate=45] at (4,3) {\dots};
    \node[fill=white,scale=0.8,rotate=45] at (4,4) {\dots};

\end{tikzpicture}
\end{center}
The number of elements on the diagonal lines where $a+b$ is constant is the same as before, 
but the number on vertical lines
is different. This time, the fixed-set map has the form
\begin{multline*}
    (-)^\GG\colon H_\GG^{\Omega_1+2\Omega_2+a+b\sigma}(B U(2)_+) \to \\
    H^{a}(BU(2)_+;\Z) \dirsum H^{a-2}(BT^2_+;\Z) \dirsum H^{a-4}(BU(2)_+;\Z).
\end{multline*}
So, for example, the five basis elements in gradings of the form $\Omega_1+2\Omega_2+4+b\sigma$ give a basis of
\[
    H^{4}(BU(2)_+;\Z) \dirsum H^{2}(BT^2_+;\Z) \dirsum H^{0}(BU(2)_+;\Z),
\]
for which a more familiar basis would be
\[
    \{ (c_1^2,0,0), (c_2,0,0), (0,x_1,0), (0,x_2,0), (0,0,1) \}.
\]

Finally, an example to dispel the idea that may be raised by the preceding two examples,
that the basis elements always lie along three diagonally rising lines.
A similar analysis in the gradings $\Omega_0+RO(\GG)$ shows the following
arrangement of basis elements.
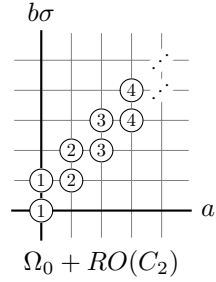
\begin{center}
\begin{tikzpicture}[scale=0.4] 
    \draw[step=1cm,gray,very thin] (-0.9,-0.9) grid (4.9,5.9);
    \draw[thick] (-1,0) -- (5,0);
    \draw[thick] (0,-1) -- (0,6);
    \node[right] at (5,0) {$a$};
    \node[above] at (0,6) {$b\sigma$};
    \node[below] at (2,-1) {$\Omega_0 + RO(\GG)$};

    \node[fill=white,scale=0.7] at (0,0) {1}; \draw (0,0) circle(0.38cm);
    \node[fill=white,scale=0.7] at (0,1) {1}; \draw (0,1) circle(0.38cm);

    \node[fill=white,scale=0.7] at (1,1) {2}; \draw (1,1) circle(0.38cm);
    \node[fill=white,scale=0.7] at (1,2) {2}; \draw (1,2) circle(0.38cm);

    \node[fill=white,scale=0.7] at (2,2) {3}; \draw (2,2) circle(0.38cm);
    \node[fill=white,scale=0.7] at (2,3) {3}; \draw (2,3) circle(0.38cm);

    \node[fill=white,scale=0.7] at (3,3) {4}; \draw (3,3) circle(0.38cm);
    \node[fill=white,scale=0.7] at (3,4) {4}; \draw (3,4) circle(0.38cm);

    \node[fill=white,scale=0.8,rotate=45] at (4,4) {\dots};
    \node[fill=white,scale=0.8,rotate=45] at (4,5) {\dots};

\end{tikzpicture}
\end{center}

\section{Units and dual elements}\label{sec:units}

We now determine the group of units of $H_\GG^\gr(BU(2)_+)$.
We begin with the following, which narrows down where we need to look.

\begin{proposition}
The only units of $H_\GG^\gr(BU(2)_+)$ live in grading $0$.
\end{proposition}

\begin{proof}
Suppose that $u$ is a unit and lives in grading
\[
    \gamma = \grad u = m_0\Omega_0 + m_1\Omega_1 + m_2\Omega_2 + \alpha
\]
where $\alpha\in RO(\GG)$.
Because $\rho(u)$ must be a unit in $H^\Z(BU(2)_+)$, so $\rho(u) = \pm 1\in H^0(BU(2)_+)$,
we must have $\rho(\alpha) = 0$,
hence $\alpha = k(1-\sigma)$ for some $k\in\Z$.
That is, $u$ is on the ``0'' diagonal line in $H_\GG^{\gamma+RO(\GG)}(BU(2)_+)$.

Let $m = \min\{m_0,m_1,m_2\}$. There is one basis element in $H_\GG^{\gamma+RO(\GG)}(BU(2)_+)$
on the 0 diagonal line, which is $z = \zeta_0^{m_0-m}\zeta_1^{m_1-m}\zeta_2^{m_2-m}$.
All other basis elements are above and to the right of that, strictly so in at least one direction.

If $u$ were strictly above and to the left of $z$, then, when we express $u$ in terms of the basis,
only $z$ could appear, and we would have $u = a\xi^n z$ for some $a\in\Z$ and $n\geq 1$.
But then we would have
\[
    u^\GG = (a\xi^n z)^\GG = 0,
\]
because $\xi^\GG = 0$,
while $u^\GG$ must be a unit because $u$ is. Hence $u$ cannot be strictly above and to the left of $z$.

If $u$ were strictly below and to the right of $z$, again think of writing $u$ in terms
of the basis. Several basis elements may be involved in that linear combination, but the
coefficients of all except $z$ would vanish on applying $\rho$. The coefficient
of $z$ would have to have the form $a\tau(\iota^{-n})$ for some $a\in\Z$ and $n\geq 1$.
But
\[
    \rho(\tau(\iota^{-n})) =
    \begin{cases}
        2 & \text{if $n$ is even} \\
        0 & \text{if $n$ is odd}
    \end{cases}
\]
so we cannot have $\rho(u) = \pm 1$ as we need because $u$ is a unit.

Hence, $u$ must be in the same grading as $z$, which implies that $k = 2m$.

Applying the same argument to $u^{-1}$,
which lives in grading $-m_0\Omega_0-m_1\Omega_1-m_2\Omega_2-k(1-\sigma)$,
we see that $k = 2M$ where
$M = \max\{m_0,m_1,m_2\}$. Hence, $m = M = m_0 = m_1 = m_2$ and $\gamma = 0$.
That is, $u$ must live in grading 0.
\end{proof}

Now we identify the units in $H_\GG^0(BU(2)_+)$.

\begin{definition}
Let
\begin{align*}
    \epsilon_\lambda &= e^{-2}\kappa \zeta_0\zeta_2 \cl \\
\intertext{and}
    \epsilon_\omega &= e^{-4}\kappa\zeta_0^2\zeta_1\cw.
\end{align*}
\end{definition}

One reason for picking out these elements is the following:
If we look at the basis for $H_\GG^{RO(\GG)}(BU(2)_+)$ given in \S\ref{sec:bases},
we can see that
\[
    H_\GG^0(BU(2)_+) \iso A(\GG)\dirsum\Z\dirsum\Z
\]
with generators $1$, $\epsilon_\lambda$, and $\epsilon_\omega$.
We can rewrite this as $\Z^4$ with generators $1$, $g$, $\epsilon_\lambda$, and $\epsilon_\omega$.

\begin{proposition}
There are exactly 16 units in $H_\GG^\gr(BU(2)_+)$, given by
\[
    \pm(1-\kappa)^{a_1}(1-\epsilon_\lambda)^{a_2}(1-\epsilon_\omega)^{a_3}
\]
where $a_1, a_2, a_3\in \{0,1\}$.
\end{proposition}

\begin{proof}
The preceding proposition tells us that any units must live in $H_\GG^0(BU(2)_+)$,
so we just need to find the units in that ring.
We first verify that the elements listed in the proposition are units.

From the fact that $g\cdot e^{-n}\kappa = 0$ we get $g\epsilon_\lambda = 0$ and
$g\epsilon_\omega = 0$.
We also have the following relations:
\begin{align*}
    \epsilon_\lambda^2 &= 2\epsilon_\lambda \\
    \epsilon_\omega^2 &= 2\epsilon_\omega \\
    \epsilon_\lambda \epsilon_\omega &= 0.
\end{align*}
The first line follows from the similar result in the cohomology of $BU(1)$
given in \cite[Proposition~11.5]{Co:InfinitePublished}.
For the second, we compute
\begin{align*}
    \epsilon_\omega^2
    &= (e^{-4}\kappa\zeta_0^2\zeta_1\cw)^2 \\
    &= 2e^{-8}\kappa \zeta_0^2\zeta_1^2\cw \cdot \zeta_0^2\cw \\
    &= 2e^{-8}\kappa \zeta_0^2\zeta_1^2\cw((1-\kappa)\zeta_2^2\cxw + e^2\cxl) \\
    &= 2e^{-6}\kappa \zeta_0^2\zeta_1\cw \cdot \zeta_1\cxl \\
    &= 2e^{-6}\kappa \zeta_0^2\zeta_1\cw((1-\kappa)\zeta_0\zeta_2\cl + e^2) \\
    &= 2e^{-4}\kappa \zeta_0^2\zeta_1\cw \\
    &= 2\epsilon_\omega.
\end{align*}
For the vanishing of the product, we compute
\[
    \epsilon_\lambda \epsilon_\omega
    = e^{-2}\kappa \zeta_0\zeta_2 \cl \cdot e^{-4}\kappa\zeta_0^2\zeta_1\cw 
    = 2e^{-6}\kappa \xi \zeta_0^2 \cl\cw 
    = 0.
\]
With the fact that $\kappa^2 = 2\kappa$, we then see that each of $1-\kappa$,
$1-\epsilon_\lambda$, and $1-\epsilon_\omega$ squares to 1, hence the elements
listed in the proposition are indeed units.

Now we show that these are the only units.
As mentioned before the proposition, $H_\GG^0(BU(2)_+)\iso \Z^4$
with generators $1$, $g$, $\epsilon_\lambda$, and $\epsilon_\omega$,
so any units would have the form
\[
    u = a_1 + a_2g + a_3\epsilon_\lambda + a_4\epsilon_\omega
\]
for some integers $a_i\in\Z$.
We set the product of two of these equal to 1:
\begin{align*}
    1 &= (a_1 + a_2g + a_3\epsilon_\lambda + a_4\epsilon_\omega)
          (b_1 + b_2g + b_3\epsilon_\lambda + b_4\epsilon_\omega) \\
    &= a_1b_1 + (a_1b_2 + a_2b_1 + 2a_2b_2)g \\
    &\qquad\qquad + (a_1b_3 + a_3b_1 + 2a_3b_3)\epsilon_\lambda \\
    &\qquad\qquad + (a_1b_4 + a_4b_1 + 2a_4b_4)\epsilon_\omega
\end{align*}
Because $1$, $g$, $\epsilon_\lambda$, and $\epsilon_\omega$ are linearly independent,
we must have $a_1b_1 = 1$, hence $a_1 = b_1 = \pm 1$. Consider the case where they are equal to 1.
Then we have
\[
    b_2 + a_2 + 2a_2b_2 = 0.
\]
Multiplying by 2 and adding 1 gives
\[
    (1 + 2a_2)(1 + 2b_2) = 1.
\]
Since the two factors are integers, we must have $1 + 2a_2 = 1 + 2b_2 = \pm 1$,
from which we get $a_2 = b_2 = 0$ or $-1$.

Similarly, still assuming $a_1 = b_1 = 1$, we must have $a_3 = b_3 = 0$ or $-1$ and
$a_4 = b_4 = 0$ or $-1$.
If $a_1 = b_1 = -1$, we would get $a_2 = b_2 = 0$ or $1$, and similarly for the others.

Let $z_i = 0$ or $1$ for $i=1$, 2, and 3, then our units are the 16 elements
\[
    \pm(1 - z_1g - z_2\epsilon_\lambda - z_3\epsilon_\omega)
    = \pm(1-z_1g)(1-z_2\epsilon_\lambda)(1-z_3\epsilon_\omega).
\]
Rewriting $1-g = -(1-\kappa)$, we get the units listed in the proposition.
\end{proof}

\begin{remark}
We noted in the proof that $1-g = -(1-\kappa)$, so we could use either notation for this unit.
We tend to prefer $1-\kappa$ because it restricts to 1 nonequivariantly, whereas $1-g$ restricts to $-1$.
\end{remark}

For many purposes, it is useful to look at generators
associated to the dual bundle $\omega\dual$ rather than $\omega$ itself.
There is a $\GG$-involution
\[
    \delta\colon BU(2)\to BU(2)
\]
classifying $\omega\dual$, and we let
\begin{align*}
    \cwd &= \delta^*\cw = e(\omega\dual) \\
    \cxwd &= \delta^*\cxw = e(\chi\omega\dual) \\
    \cld &= \delta^*\cl = e(\lambda\dual) \\
\intertext{and}
    \cxld &= \delta^*\cxl = e(\chi\lambda\dual).
\end{align*}
Because $\delta^*$ is an algebra isomorphism, these elements can be used as generators
in place of $\cw$, etc.,
satisfying the relations listed in Theorem~\ref{thm:main}, \textit{mutatis mutandis.}

On the other hand, we should be able to write these dual elements in terms of the original
generators and we now work out those expressions.
Recall that, nonequivariantly, the Chern classes of the dual of the tautological bundle are
\[
    \cd_2 = c_2 \qquad\text{and}\qquad \cd_1 = -c_1.
\]
Here are the analogous results in the equivariant case.

\begin{proposition}
The dual classes can be written in terms of the usual generators as
\begin{align*}
    \cwd &= (1-\epsilon_\lambda)\cw \\
    \cxwd &= (1-\epsilon_\lambda)\cxw \\
    \cld &= -(1-\epsilon_\lambda)\cl \\
    \cxld &= -(1-\kappa)(1-\epsilon_\lambda)\cxl.
\end{align*}
\end{proposition}

\begin{proof}
The formulas for $\cld$ and $\cxld$ follow from the similar formulas
in the cohomology of $BU(1)$ given in \cite[Proposition~11.6]{Co:InfinitePublished},
by pulling back along the classifying map for $\lambda$.

For $\cwd$ and $\cxwd$, these elements are in even gradings,
where we know that $\eta$ is injective. By an argument similar to
that used in Proposition~\ref{prop:fixedformulas} and in \cite{Co:InfinitePublished}, we have
\begin{align*}
    \eta(\cwd) &= (c_2\zeta_1\zeta_2^2, -x_2(e^2-\xi x_2)\zeta_0^{-1}\zeta_2,
                    (e^4-e^2\xi c_1 + \xi^2c_2)\zeta_0^{-2}\zeta_1^{-1}) \\
        &= \eta(\cw) - (0, 2e^2x_2, 0) \\
        &= \eta(\cw) - \eta(e^{-2}\kappa \zeta_0\zeta_2\cl\cw) \\
\intertext{using that $2e^2\xi = 0$, so}
        \cwd &= \cw - e^{-2}\kappa \zeta_0\zeta_2\cl\cw \\
            &= (1-\epsilon_\lambda)\cw.
\end{align*}
When we apply $\chi$, $\cl$ is fixed because $\ext^2(\chi\omega) = \ext^2\omega$, so we get
\[
    \cxwd = (1-\epsilon_\lambda)\cxw
\]
as well.
\end{proof}

Note that $\rho(1-\epsilon_\lambda) = 1$ and $\rho(1-\kappa) = 1$, so these relations restrict
to the nonequivariant ones recalled just before the proposition.

The units can be rewritten in terms of the dual classes using the following.

\begin{corollary}
We have
\begin{align*}
    \epsilon_\lambda &= e^{-2}\kappa\zeta_0\zeta_2\cld \\
\intertext{and}
    \epsilon_\omega &= e^{-4}\kappa\zeta_0^2\zeta_1\cwd.
\end{align*}
\end{corollary}

\begin{proof}
\begin{align*}
    e^{-2}\kappa\zeta_0\zeta_2\cld
    &= -e^{-2}\kappa\zeta_0\zeta_2\cdot (1-\epsilon_\lambda)\cl \\
    &= -\epsilon_\lambda(1-\epsilon_\lambda) \\
    &= \epsilon_\lambda
\end{align*}
using that $\epsilon_\lambda^2 = 2\epsilon_\lambda$, and
\begin{align*}
    e^{-4}\kappa\zeta_0^2\zeta_1\cwd
    &= e^{-4}\kappa\zeta_0^2\zeta_1\cdot (1-\epsilon_\lambda)\cw \\
    &= \epsilon_\omega(1-\epsilon_\lambda) \\
    &= \epsilon_\omega
\end{align*}
using that $\epsilon_\lambda\epsilon_\omega = 0$.
\end{proof}

\section{Characteristic numbers of lines and surfaces}\label{sec:char numbers}

Nonequivariantly, once we know 
the cohomology of $BU(n)$ we may define characteristic classes
for $n$-dimensional (complex) vector bundles.
From those, we can define
characteristic numbers of stably almost complex manifolds.
These are known to characterize the (nonequivariant) bordism classes of such manifolds.

Equivariantly, we make the following definition.

\begin{definition}\label{def:char numbers}
Let $M$ be a complex $\GG$-manifold of (complex) dimension $n$, so that its tangent bundle $\tau_M$
has a given complex structure. Let 
\[
    \tau_M\colon M\to B U(n)
\]
be the classifying map for the tangent bundle,
and let $c\in H_\GG^{\omega(n)+RO(\GG)}(B U(n)_+)$ be a cohomology class. We call
\[
    c[M] := \eval{\tau_M^*(c)}{[M]} \in \HS
\]
a \emph{(tangential) characteristic number} of $M$.
(Here, $\eval{-}{-}$ denotes evaluation of cohomology on homology, as in
\cite[Definition~3.10.18]{CostenobleWanerBook}.)
\end{definition}

A characteristic ``number'' here will be an element of the cohomology of a point rather than an integer.
One particular example is the \emph{Euler characteristic} of a manifold, which is
\[
    c_{\omega(n)}[M] = \eval{\tau_M^*(c_{\omega(n)})}{[M]} = \eval{e(\tau_M)}{[M]} \in \HS^0 = A(\GG).
\]
The nonequivariant proof that characteristic numbers are cobordism invariants generalizes
to show that equivariant characteristic numbers are equivariant cobordism invariants.

As examples, we compute the characteristic numbers for the (complex)
lines $\Xpq pq$ with $p+q = 2$ and the
surfaces $\Xpq pq$ with $p+q = 3$.
We will use the notations from and computations of the cohomologies of these spaces
made in \cite{CHTFiniteProjSpace}.
Note that $\Xp 2 \homeo \Xq 2$, $\Xp 3 \homeo \Xq 3$ and $\Xpq 2{} \homeo \Xpq 12$, so there are really only two
lines and two surfaces to consider.

\subsection{Characteristic numbers for lines}
We begin by identifying the relevant characteristic classes $c$ for complex lines.
As in Definition~\ref{def:char numbers}, they will be elements in $H_\GG^{\omega(1)+RO(\GG)}(B U(1)_+)$.
Write $\omega = \omega(1)$ for the remainder of this subsection.

Rewrite
\[
    \eval{\tau_M^*(c)}{[M]} = \eval{c}{(\tau_M)_*[M]},
\]
and notice that $\grad(\tau_M)_*[M] = \omega$,
so we need only consider classes $c$ that can give non-zero elements of the cohomology of a point
when evaluated on a homology class in grading $\omega$. These will, first of all, be elements
in gradings in the coset $\omega+RO(\GG)$.

In \cite{Co:InfinitePublished}, the first author gave a basis for the cohomology of $B U(1)$
which, in gradings $\omega+RO(\GG) = \Omega_1 + RO(\GG)$, begins
\[
    \{ \zeta_1, \cw, \zeta_0\cw^2, \cw^2\cxw, \zeta_0\cw^3\cxw, \dots \}.
\]
These are arranged in a step pattern as in the following diagram.
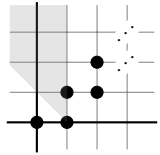
\begin{center}
\begin{tikzpicture}[scale=0.4] 
    \draw[step=1cm,gray,very thin] (-0.9,-0.9) grid (3.9,3.9);
    \draw[thick] (-1,0) -- (4,0);
    \draw[thick] (0,-1) -- (0,4);

    \draw[fill] (0,0) circle(0.2cm);
    \draw[fill] (1,0) circle(0.2cm);
    \draw[fill] (1,1) circle(0.2cm);
    \draw[fill] (2,1) circle(0.2cm);
    \draw[fill] (2,2) circle(0.2cm);

    \node[fill=white,scale=0.8,rotate=45] at (3,2) {\dots};
    \node[fill=white,scale=0.8,rotate=45] at (3,3) {\dots};

    \fill[gray,opacity=0.2] (1,0) -- (1,3.9) -- (-0.9,3.9) -- (-0.9,1.9) -- (1,0);

\end{tikzpicture}
\end{center}
The only basis elements that can evaluate to something nonzero on an element in grading $\omega$
are those that fall into the shaded positive wedge above that grading
(there are no basis elements in the negative wedge below it).
The two elements in that wedge are $\cw$ and $\zeta_0\cw^2$, so those are the only two
characteristic numbers we need to compute.
All other characteristic numbers of a line will be linear combinations of these two.

\subsection{Characteristic numbers of $\Xp 2$}
This line has trivial $\GG$-action. Recall that, grading on $RO(\Pi BU(1))$, its cohomology is
\[
    H_\GG^{RO(\Pi BU(1))}(\Xp 2_+) \iso \HS[\cwd,\zeta_1^{\pm 1}]/\rels{\cwd[2]}.
\]
Because the action of $\GG$ is trivial, its tangent bundle $\tau_2$ has (real) dimension 2,
which is the grading in which its fundamental class $[\Xp 2]$ lives.
We will therefore be evaluating cohomology elements in the $RO(\GG)$ grading on $[\Xp 2]$.
A basis in the $RO(\GG)$ grading is given by the two elements $\{1, \zeta_1^{-1}\cwd\}$
and we now evaluate each on the fundamental class.
(We use the dual classes here because, as Milnor and Stasheff put it,
they are the generators ``compatible with the preferred orientation'' of the projective spaces
\cite[\S14.10]{MilnorStasheff}.
They will arise below as the pullbacks of the usual characteristic classes along the
classifying map of the tangent bundle.)

We have $\eval{1}{[\Xp 2]} = 0$ for dimensional reasons---it lives in $\HS^{-2} = 0$.

We must have $\eval{\zeta_1^{-1}\cwd}{[\Xp 2]} = a + b\kappa \in A(\GG)$ for some
integers $a$ and $b$.
Applying $\rho$,
\[
    \rho(a+b\kappa) = a
\]
but
\[
    \rho\eval{\zeta_1^{-1}\cwd}{[\Xp 2]} = \eval{\cd_1}{[\Xp 2]} = 1
\]
when we use the dual class $\cd_1$.
Therefore, $a = 1$. Taking fixed points, we have
\[
    (a+b\kappa)^\GG = a + 2b = 1 + 2b
\]
but
\[
    \eval{\zeta_1^{-1}\cwd}{[\Xp 2]}^\GG = \eval{\cd_1}{[\Xp 2]} = 1
\]
again, hence $b = 0$. Therefore,
\[
    \eval{\zeta_1^{-1}\cwd}{[\Xp 2]} = 1.
\]

Let $\tau_2\colon \Xp 2\to B U(1)$ be the classifying map of the tangent bundle.
In order to calculate characteristic numbers, we need to calculate $\tau_2^*(\cw)$
and $\tau_2^*(\zeta_0^2\cw^2)$. We first notice that the induced map on representation rings has
\[
    \tau_2^*(\Omega_0) = 2\sigma - 2 \quad\text{and}\quad \tau_2^*(\Omega_1) = 0,
\]
which tells us that
\[
    \tau_2^*(\zeta_0) = \xi \quad\text{and}\quad \tau_2^*(\zeta_1) = 1.
\]
To calculate $\tau_2^*(\cw)$, we know this element lives in grading 2, hence
\[
    \tau_2^*(\cw) = (a+b\kappa)\zeta_1^{-1}\cwd
\]
for some $a,b \in \Z$. The restriction $\rho\tau_2^*(\cw)$ is the nonequivariant Euler class of the tangent
bundle, which is $2\cd_1$, hence $a = 2$.
Taking fixed sets must give us the same Euler class, so $a + 2b = 2$ as well, which tells us that $b = 0$.
Hence, $\tau_2^*(\cw) = 2\zeta_1^{-1}\cwd$.

From this, we can compute
\[
    \tau_2^*(\zeta_0\cw^2) = \xi\cdot 4\zeta_1^{-2}\cwd[2] = 0.
\]

Finally, we can compute the two characteristic numbers:
\begin{equation}\label{eqn:char numbers 20}
\begin{aligned}
    \cw{[\Xp 2]} &= \eval{2\zeta_1^{-1}\cwd}{[\Xp 2]} = 2 \\
    \zeta_0\cw^2[\Xp 2] &= \eval{0}{[\Xp 2]} = 0.
\end{aligned}
\end{equation}

\subsection{Characteristic numbers of $\Xpq 1{}$}
This line has nontrivial action, and its dimension is $2\sigma$.
That means that we are again interested in cohomology elements in the $RO(\GG)$ grading,
where we have the basis $\{1, \zeta_0\cwd\}$.
Again, we start by evaluating each of these on the fundamental class $[\Xpq 1{}]$.
\[
    \eval{1}{[\Xpq 1{}]} = a e^{-2}\kappa \in \HS^{-2\sigma}
\]
for some $a\in\Z$. Applying $\rho$ will give 0 on both sides, so we look instead at
fixed points:
\[
    (ae^{-2}\kappa)^\GG = 2a,
\]
and
\[
    \eval{1}{[\Xpq 1{}]}^\GG = \eval{1}{[\Xp{}]} + \eval{1}{[\Xq{}]} = 1 + 1 = 2.
\]
Therefore, $a = 1$ and $\eval{1}{[\Xpq 1{}]} = e^{-2}\kappa$.

\[
    \eval{\zeta_0\cwd}{[\Xpq 1{}]} = a + b\kappa \in A(\GG)
\]
for some $a,b\in Z$.
Nonequivariantly, we have
\[
    \rho\eval{\zeta_0\cwd}{[\Xpq 1{}]} = \eval{\cd_1}{[\PP^1]} = 1,
\]
so we must have $a = 1$. Taking fixed points,
\[
    \eval{\zeta_0\cwd}{[\Xpq 1{}]}^\GG = \eval{0}{[\Xp{}]} + \eval{1}{[\Xq{}]} = 0 + 1 = 1,
\]
hence $b = 0$ and $\eval{\zeta_0\cwd}{[\Xpq 1{}]} = 1$.

Let $\tau_{1,1}\colon\Xpq 1{}\to B U(1)$ be the classifying map for the tangent bundle.
We have
\[
    \tau_{1,1}^*(\Omega_0) = 0 \quad\text{and}\quad \tau_{1,1}^*(\Omega_1) = 2\sigma - 2,
\]
so
\[
    \tau_{1,1}(\zeta_0) = 1 \quad\text{and}\quad \tau_{1,1}^*(\zeta_1) = \xi.
\]
We must have
\[
    \tau_{1,1}^*(\cw) = \alpha \zeta_0\cwd + be^2
\]
for some $\alpha\in A(\GG)$ and $b\in\Z$. Applying $\rho$, we get
\[
    \rho(\alpha \zeta_0\cwd + be^2) = \rho(\alpha)\cd_1
\]
and
\[
    \rho\tau_{1,1}^*(\cw) = c_1(\tau_{1,1}) = 2\cd_1,
\]
the Euler class. Hence, $\rho(\alpha) = 2$.
Taking fixed points, we get
\[
    (\alpha \zeta_0\cwd + be^2)^\GG = \alpha^\GG(0,1) + b(1,1)
\]
and
\[
    \tau_{1,1}^*(\cw)^\GG = (1,1),
\]
hence $\alpha^\GG = 0$ and $b = 1$. This tells us that $\alpha = g$, hence
\[
    \tau_{1,1}^*(\cw) = g \zeta_0\cwd + e^2.
\]
We can then compute
\[
    \tau_{1,1}^*(\zeta_0\cw^2) = (g \zeta_0\cwd + e^2)^2 = e^4
\]
using the relations in the cohomology of $\Xpq 1{}$.

We can now put this all together to compute the two characteristic numbers:
\begin{equation}\label{eqn:char numbers 11}
\begin{aligned}
    \cw{[\Xpq 1{}]} &= \eval{g \zeta_0\cwd + e^2}{[\Xpq 1{}]} = g + \kappa = 2 \\
    \zeta_0\cw^2[\Xpq 1{}] &= \eval{e^4}{[\Xpq 1{}]} = 2e^2.
\end{aligned}
\end{equation}

\begin{corollary}
The lines $\Xp 2$ and $\Xpq 1{}$ are not equivariantly cobordant.
\end{corollary}

\begin{proof}
Their characteristic numbers, computed in (\ref{eqn:char numbers 20}) and (\ref{eqn:char numbers 11}),
are different.
\end{proof}

This result is not suprising if we look at fixed points, but it's interesting that
the characteristic numbers detect the difference. Also note that the first number
is the same for both lines (and is just the Euler characteristic), 
so we need the second as well (which reduces to 0 nonequivariantly), contrary to the nonequivariant
case where the number $\cd_1[M]$ suffices to determine the cobordism class of a 
complex one-dimensional manifold $M$.

\subsection{Characteristic numbers for surfaces}
As we did for lines,
we begin by identifying the relevant characteristic classes for a surface. As in Definition~\ref{def:char numbers},
they will be elements $c\in H_\GG^{\omega(2)+RO(\GG)}(B U(2)_+)$.
In the remainder of this section we will write $\omega$ for $\omega(2)$
or for the tautological line bundle over a finite projective space, and the context should
make clear which we mean.

Once again, we need only consider classes $c$ that can give non-zero elements of the cohomology of a point
when evaluated on a homology class in grading $\omega$.
In \S\ref{sec:bases}, we listed the first few elements of the basis in gradings 
$\omega+RO(\GG) = \Omega_1 + 2\Omega_2 + RO(\GG)$
given by Theorem~\ref{thm:main}.
The ones we need are those in the wedge above position $\omega = \Omega_1+2\Omega_2 + 4$
shown shaded below, corresponding to the positive wedge in the cohomology of a point.
(There are no basis elements in the negative wedge below grading $\omega$.)
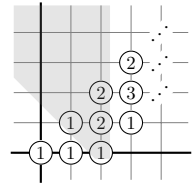
\begin{center}
\begin{tikzpicture}[scale=0.4] 
    \draw[step=1cm,gray,very thin] (-0.9,-0.9) grid (4.9,4.9);
    \draw[thick] (-1,0) -- (5,0);
    \draw[thick] (0,-1) -- (0,5);

    \node[fill=white,scale=0.7] at (0,0) {1}; \draw (0,0) circle(0.38cm);

    \node[fill=white,scale=0.7] at (1,0) {1}; \draw (1,0) circle(0.38cm);
    \node[fill=white,scale=0.7] at (1,1) {1}; \draw (1,1) circle(0.38cm);

    \node[fill=white,scale=0.7] at (2,0) {1}; \draw (2,0) circle(0.38cm);
    \node[fill=white,scale=0.7] at (2,1) {2}; \draw (2,1) circle(0.38cm);
    \node[fill=white,scale=0.7] at (2,2) {2}; \draw (2,2) circle(0.38cm);

    \node[fill=white,scale=0.7] at (3,1) {1}; \draw (3,1) circle(0.38cm);
    \node[fill=white,scale=0.7] at (3,2) {3}; \draw (3,2) circle(0.38cm);
    \node[fill=white,scale=0.7] at (3,3) {2}; \draw (3,3) circle(0.38cm);

    \node[fill=white,scale=0.8,rotate=45] at (4,2) {\dots};
    \node[fill=white,scale=0.8,rotate=45] at (4,3) {\dots};
    \node[fill=white,scale=0.8,rotate=45] at (4,4) {\dots};

    \fill[gray,opacity=0.2] (2.3,-0.2) -- (2.3,4.9) -- (-0.9,4.9) -- (-0.9,2.5) -- (1.8,-0.2) -- (2.3,-0.2);

\end{tikzpicture}
\end{center}
Thus, we are interested in the following six basic elements, which live in the listed gradings:
\begin{equation}\label{eqn:char classes}
\begin{aligned}
    & c &\qquad& \grad c \\
    &\zeta_0\zeta_2^3\cl^2 && \omega - 2 + 2\sigma \\
    &\cw && \omega \\
    &\zeta_2^2\cl^2\cxl && \omega + 2\sigma \\
    &\zeta_0\zeta_2\cl\cw && \omega + 2\sigma \\
    &\zeta_0^2\zeta_1\cw^2 && \omega + 4\sigma \\
    &\zeta_0\zeta_2^3\cl^3\cxl && \omega + 4\sigma
\end{aligned}
\end{equation}
It suffices, for any particular surface $M$, to calculate the characteristic numbers
determined by these classes, as all nonzero characteristic numbers will be linear combinations of these six.

\subsection{Characteristic numbers of $\Xp 3$}

This surface has trivial $\GG$-action.
Recall that, if we grade on $RO(\Pi B) = RO(\Pi BU(1))$, we get
\[
    H_\GG^{RO(\Pi B)}(\Xp 3_+) \iso \HS[\cwd,\zeta_1^{\pm 1}] / \rels{\cwd[3]}.
\]
As the action is trivial, the tangent bundle $\tau_3$ has (real) dimension 4,
and this is the grading in which the fundamental class $[\Xp 3]$ lives.
We will therefore want to evaluate cohomology elements in the $RO(\GG)$ grading on $[\Xp 3]$.
A basis in this grading is given by
\[
    \{ 1, \zeta_1^{-1}\cwd, \zeta_1^{-2}\cwd[2] \}.
\]
The calculations are similar to those we did for lines; we suppress the details
and just give the results:
\begin{align*}
    \eval{1}{[\Xp 3]} &= 0 \\
    \eval{\zeta_1^{-1}\cwd}{[\Xp 3]} &= 0 \\
    \eval{\zeta_1^{-2}\cwd[2]}{[\Xp 3]} &= 1.
\end{align*}
If $\tau_3\colon \Xp 3\to B U(2)$ is the classifying map of the tangent bundle,
we get the following computations.
\begin{align*}
    \tau_3^*(\zeta_0) &= \xi \\
    \tau_3^*(\zeta_1) &= \tau_3^*(\zeta_2) = 1 \\
    \tau_3^*(\cw) &= 3\zeta_1^{-2}\cwd[2] \\
    \tau_3^*(\cl) &= 3\zeta_1^{-1}\cwd \\
    \tau_3^*(\cxl) &= e^2 + 3\xi\zeta_1^{-1}\cwd.
\end{align*}
which give the pullbacks of the elements from (\ref{eqn:char classes}):
\begin{align*}
    \tau_3^*(\zeta_0\zeta_2^3\cl^2) &= 9\xi\zeta_1^{-2}\cwd[2] \\
    \tau_3^*(\cw) &= 3\zeta_1^{-2}\cwd[2] \\
    \tau_3^*(\zeta_2^2\cl^2\cxl) &= 9e^2\zeta_1^{-2}\cwd[2] \\
    \tau_3^*(\zeta_0\zeta_2\cl\cw) &= 0 \\
    \tau_3^*(\zeta_0^2\zeta_1\cw^2) &= 0 \\
    \tau_3^*(\zeta_0\zeta_2^3\cl^3\cxl) &= 0
\end{align*}
Putting all these calculations together, we get the following characteristic numbers.
\begin{equation}\label{eqn:char numbers 30}
\begin{aligned}
    \zeta_0\zeta_2^3\cl^2[\Xp 3] &= 9\xi \\
    \cw[\Xp 3] &= 3 \\
    \zeta_2^2\cl^2\cxl[\Xp 3] &= 9e^2 \\
    \zeta_0\zeta_2\cl\cw[\Xp 3] &= 0 \\
    \zeta_0^2\zeta_1\cw^2[\Xp 3] &= 0 \\
    \zeta_0\zeta_2^3\cl^3\cxl[\Xp 3] &= 0
\end{aligned}
\end{equation}

\subsection{Characteristic numbers of $\Xpq 2{}$}

The tangent bundle of $\Xpq 2{}$ is
\[
    \tau_{2,1} \iso \Hom(\omega(1), \Cpq 2{} / \omega(1)) \iso (2\omega(1)\dual \dirsum \chi\omega(1)\dual) / \C,
\]
which has dimension
\[
    \dim\tau_{2,1} = \grad[\Xpq 2{}] = 2\omega(1) + \chi\omega(1) - 2 = \omega(1) + 2\sigma.
\]
We will therefore want to evaluate elements in the cohomology of $\Xpq 2{}$ in gradings 
$\omega(1) + RO(\GG) = \Omega_1+RO(\GG)$ on $[\Xpq 2{}]$.
The cohomology in those gradings has basis
\[
    \{ \zeta_1, \cwd, \zeta_0\cwd[2] \}.
\]

Once again suppressing the details of the computations, we get the following results.
\begin{align*}
    \eval{\zeta_1}{[\Xpq 2{}]} &= 0 \\
    \eval{\cwd}{[\Xpq 2{}]} &= e^{-2}\kappa \\
    \eval{\zeta_0\cwd[2]}{[\Xpq 2{}]} &= 1.
\end{align*}
The pullbacks of the elements from (\ref{eqn:char classes}) are
\begin{align*}
    \tau_{2,1}^*(\zeta_0\zeta_2^3\cl^2) &= 9\xi\zeta_0\cwd[2] + e^4\zeta_1 \\
    \tau_{2,1}^*(\cw) &= (3-2\kappa)\zeta_0\cwd[2] + 2e^2\cwd \\
    \tau_{2,1}^*(\zeta_2^2\cl^2\cxl) &= -3e^2\zeta_0\cwd[2] + 3e^4\cwd \\
    \tau_{2,1}^*(\zeta_0\zeta_2\cl\cw) &= -2e^2\zeta_0\cwd[2] + 2e^4\cwd \\
    \tau_{2,1}^*(\zeta_0^2\zeta_1\cw^2) &= e^4\zeta_0\cwd[2] \\
    \tau_{2,1}^*(\zeta_0\zeta_2^3\cl^3\cxl) &= -3e^4\zeta_0\cwd[2] + 3e^6\cwd
\end{align*}
These calculations give the following characteristic numbers.
\begin{equation}\label{eqn:char numbers 21}
\begin{aligned}
    \zeta_0\zeta_2^3\cl^2[\Xpq 2{}] &= 9\xi \\
    \cw[\Xpq 2{}] &= 3 \\
    \zeta_2^2\cl^2\cxl[\Xpq 2{}] &= 3e^2 \\
    \zeta_0\zeta_2\cl\cw[\Xpq 2{}] &= 2e^2 \\
    \zeta_0^2\zeta_1\cw^2[\Xpq 2{}] &= e^4 \\
    \zeta_0\zeta_2^3\cl^3\cxl[\Xpq 2{}] &= 3e^4
\end{aligned}
\end{equation}
Comparing the characteristic numbers in (\ref{eqn:char numbers 30}) and (\ref{eqn:char numbers 21}),
we get the following.

\begin{corollary}
The surfaces $\Xp 3$ and $\Xpq 2{}$ are not equivariantly cobordant.
\qed
\end{corollary}

Again, this result is not suprising if we look at fixed sets, but the point is that
the characteristic numbers do detect the difference. 
What seems to be happening is that the first two characteristic numbers,
which are the same for both surfaces, are simply giving us the
nonequivariant calculations $c_1^2[\Xp 3] = 9$ and $c_2[\Xp 3] = 3$, while the others
(which all reduce to 0 nonequivariantly) depend on the fixed-point structures, so
distinguish between the two.

\appendix

\section{Equivariant ordinary cohomology}\label{app:ordinarycohomology}

In this paper
we use $\GG$-equivariant ordinary cohomology with the extended grading developed in \cite{CostenobleWanerBook}.
This is an extension of Bredon's ordinary cohomology to be graded on representations of the fundamental groupoids
of $\GG$-spaces.
We review here some of the notation and computations we use.
More detailed summaries of this theory can be found in
\cite{Beaudry:Guide},
\cite{Co:InfinitePublished}, and \cite{CHTFiniteProjSpace}.

For an ex-$\GG$-space $Y$ over $X$, we write $H_\GG^{RO(\Pi X)}(Y;\Mackey T)$ for the
ordinary cohomology of $Y$ with coefficients in a Mackey functor $\Mackey T$, graded
on $RO(\Pi X)$, the representation ring of the fundamental groupoid of $X$.
Through most of this paper we use the Burnside ring Mackey functor $\Mackey A$ as the coefficients,
and write simply $H_\GG^{RO(\Pi X)}(Y)$.
If $X$ is simply connected and
\[
    X^\GG = X^0\disjunion X^1 \disjunion \cdots \disjunion X^n
\]
with each $X^i$ simply connected (which is true for our main example, $BU(2)$), then
\[
    RO(\Pi X) \iso \Z\{1,\sigma,\Omega_0,\dots,\Omega_n\}/\rels{\Omega_0 + \cdots + \Omega_n = 2\sigma - 2}
\]
where $1 = [\R]$ and $\sigma = [\R^\sigma]$ are the generators of $RO(\GG)$,
the proof being essentially the same as \cite[Proposition~6.1]{Co:InfinitePublished} and the discussion following.
In this case, elements of $RO(\Pi X)$ are determined by their restrictions to the components
of $X^\GG$, which are virtual representations all of the same nonequivariant dimension and whose
fixed representations all have the same parity. The particular representation $\Omega_i$
is characterized by the fact that its restriction to $X^j$ is 0 if $i\neq j$ and $2\sigma - 2$ if $i = j$.

In \cite{CostenobleWanerBook} and \cite{CHTFiniteProjSpace} we considered cohomology to be Mackey functor--valued,
which is useful for many computations, and wrote $\Mackey H_\GG^{RO(\Pi X)}(Y)$ for the resulting theory.
In this paper we concentrate on the values at level $\GG/\GG$, and write
$H_\GG^{RO(\Pi X)}(Y) = \Mackey H_\GG^{RO(\Pi X)}(Y)(\GG/\GG)$. However, we 
use extensively the structure maps of the Mackey functor structure, namely the
restriction functor $\rho$ from equivariant cohomology to nonequivariant cohomology, and the transfer
map $\tau$ going in the other direction.
We will also treat nonequivariant cohomology as graded on $RO(\Pi X)$ via the forgetful
map $RO(\Pi_\GG X)\to RO(\Pi_e X)$ from the representation ring of the equivariant fundamental groupoid
of $X$ to the representation ring of its nonequivariant fundamental groupoid.
Another way of saying that is that we view nonequivariant cohomology as
\[
    H^{RO(\Pi X)}(X;\Z) = \Mackey H^{RO(\Pi X)}_\GG(X;\Mackey A)(\GG/e).
\]
(See \S\ref{sec:change of grading} for more details about regrading graded rings like $H^\Z(X;\Z)$.)

One of the most important facts about this theory is that we have Thom isomorphisms for every
vector bundle over $X$, hence Euler classes for every vector bundle.
(See \cite[Theorem~3.11.3]{CostenobleWanerBook}.)
If $\omega$ is a vector bundle over $X$, the representations given by each fiber
determine a representation of $\Pi X$, that is, an element of $RO(\Pi X)$,
which we call the equivariant \emph{dimension} of $\omega$, and
this is the grading in which the Euler class of $\omega$ lives.
This notion of dimension was introduced in \cite{CMW:orientation}
and the idea of grading on $RO(\Pi X)$ is to encompass the dimensions
of bundles, providing natural locations in which Thom and Euler classes can live.

For all $X$ and $Y$, $H_\GG^{RO(\Pi X)}(Y)$ is a graded module over the $RO(\GG)$-graded cohomology of a point,
\[
    \HS = \HS^{RO(\GG)} = H_\GG^{RO(\GG)}(S^0).
\]
The cohomology of a point was calculated by Stong in an unpublished
manuscript and first published by Lewis in \cite{LewisCP}. We can picture the calculation as 
in Figure~\ref{fig:cohompt},
in which a group in grading $a+b\sigma$ is plotted at the point $(a,b)$, and the spacing of the grid lines
is 2 (which is more convenient for other graphs in this paper).
The square box at the origin is a copy of $A(\GG)$, the Burnside ring of $\GG$,
closed circles are copies of $\Z$, and open circles are copies of $\Z/2$.
The elements in the second quadrant are often referred to as the \emph{positive wedge},
while those in the fourth quadrant are the \emph{negative wedge.}
\begin{figure}
\begin{tikzpicture}[x=4mm, y=4mm]
% grid and axes
	\draw[step=2, gray, very thin] (-7.8, -7.8) grid (7.8, 7.8);
	\draw[thick] (-8, 0) -- (8, 0);
	\draw[thick] (0, -8) -- (0, 8);
    \node[right] at (8,0) {$a$};
    \node[above] at (0,8) {$b\sigma$};
 %   \node[below] at (0,-8) {$\HS^{RO(\GG)}$};

    \fill (-0.3, -0.3) rectangle (0.3, 0.3);
    \fill (0, -7) circle(0.2);
    \fill (0, -6) circle(0.2);
    \fill (0, -5) circle(0.2);
    \fill (0, -4) circle(0.2);
    \fill (0, -3) circle(0.2);
    \fill (0, -2) circle(0.2);
    \fill (0, -1) circle(0.2);
    \fill (0, 1) circle(0.2);
    \fill (0, 2) circle(0.2);
    \fill (0, 3) circle(0.2);
    \fill (0, 4) circle(0.2);
    \fill (0, 5) circle(0.2);
    \fill (0, 6) circle(0.2);
    \fill (0, 7) circle(0.2);

    \fill (-2, 2) circle(0.2);
    \fill (-4, 4) circle(0.2);
    \fill (-6, 6) circle(0.2);

    \draw[fill=white] (-2, 3) circle(0.2);
    \draw[fill=white] (-2, 4) circle(0.2);
    \draw[fill=white] (-2, 5) circle(0.2);
    \draw[fill=white] (-2, 6) circle(0.2);
    \draw[fill=white] (-2, 7) circle(0.2);
    \draw[fill=white] (-4, 5) circle(0.2);
    \draw[fill=white] (-4, 6) circle(0.2);
    \draw[fill=white] (-4, 7) circle(0.2);
    \draw[fill=white] (-6, 7) circle(0.2);

    \fill (2, -2) circle (0.2);
    \fill (4, -4) circle (0.2);
    \fill (6, -6) circle (0.2);
    \draw[fill=white] (3, -3) circle(0.2);
    \draw[fill=white] (5, -5) circle(0.2);
    \draw[fill=white] (7, -7) circle(0.2);
    
    \draw[fill=white] (3, -4) circle(0.2);
    \draw[fill=white] (3, -5) circle(0.2);
    \draw[fill=white] (3, -6) circle(0.2);
    \draw[fill=white] (3, -7) circle(0.2);
    \draw[fill=white] (5, -6) circle(0.2);
    \draw[fill=white] (5, -7) circle(0.2);

    \node[right] at (0,1) {$e$};
    \node[below left] at (-2,2) {$\xi$};
    \node[above right] at (2,-2) {$\tau(\iota^{-2})$};
    \node[left] at (0,-1) {$e^{-1}\kappa$};

    \draw[fill] (11,1) rectangle (11.6,1.6); \node[right] at (11.3,1.3) {${}= A(\GG)$};
    \draw[fill] (11.3,0) circle(0.2) node[right] {${}=\Z$};
    \draw[fill=white] (11.3,-1.3) circle(0.2) node[right] {${}=\Z/2$};

\end{tikzpicture}
\caption{$\HS = H_\GG^{RO(\GG)}(S^0)$}\label{fig:cohompt}
\end{figure}

Recall that $A(\GG)$ is the Grothendieck group of finite $\GG$-sets, with multiplication given by products of sets.
Additively, it is free abelian on the classes of the orbits of $\GG$, for which we write
$1 = [\GG/\GG]$ and $g = [\GG/e]$. The multiplication is given by $g^2 = 2g$.
We also write $\kappa = 2 - g$. Other important elements in $\HS$ are shown in the figure:
The group in degree $\sigma$ is generated by an element $e$,
which can be thought of as the Euler class of $\R^\sigma\to *$,
while the group in degree $-2 + 2\sigma$ is generated by an element $\xi$.
The groups in the second quadrant are generated by the products $e^m\xi^n$, with $2e\xi = 0$.
We have $g\xi = 2\xi$ and $ge = 0$ (so $\kappa e = 2e$).
The groups in gradings $-m\sigma$, $m\geq 1$, are generated by elements $e^{-m}\kappa$, so named
because $e^m\cdot e^{-m}\kappa = \kappa$. We also have $ge^{-m}\kappa = 0$.

To explain $\tau(\iota^{-2})$, we think for moment about the nonequivariant cohomology
of a point. If we grade it on $RO(\GG)$, we get
$H^{RO(\GG)}(S^0;\Z) \iso \Z[\iota^{\pm 1}]$, where $\deg \iota = -1 + \sigma$.
(Nonequivariantly, we cannot tell the difference between $\R$ and $\R^\sigma$.)
We have $\rho(\xi) = \iota^2$ and $\tau(\iota^2) = g\xi = 2\xi$.
Note also that $\tau(1) = g$.
In the fourth quadrant we have that the group in grading $n(1-\sigma)$, $n\geq 2$, is
generated by $\tau(\iota^{-n})$.
The remaining groups in the fourth quadrant will not concern us here.
For more details, see \cite{Co:InfinitePublished} or \cite{CHTFiniteProjSpace}.

\section{Resolution of ambiguities}\label{app:resolution}

We give here the verification that we can resolve all ambiguities in the reduction system
used to prove Proposition~\ref{prop:freeness}. 
That is, we think of each reduction $W\mapsto f$ as a way of rewriting a monomial
$AW$ as $Af$, which we extend to a way of rewriting polynomials.
To resolve ambiguities means that, if a monomial can be written both as $AW_1$ and $BW_2$,
and we reduce to $Af_1$ and $Bf_2$, respectively, then we can apply a sequence of further reductions to each
that lead to the same polynomial at the end.

So we examine each pair of reductions $(W_1,f_1)$ and $(W_2,f_2)$.
It suffices to begin with the least common multiple of $W_1$ and $W_2$, apply the two reductions
to this monomial, and  show that a series of further reductions can be
applied to bring the polynomials to the same place.
One simplification we can make: If $W_1$ and $W_2$ have greatest common divisor 1,
then the verification is trivial: Applying one reduction and then the other in either order leads
to the same polynomial. Thus, we can take those pairings as verified. We now list all the pairings
together with the resolutions of their ambiguities.
The numbering of the reductions is as in the proof of Proposition~\ref{prop:freeness}.
\begin{itemize}

\item \ref{red:1} and \ref{red:2}: The least common multiple is $\zeta_0\zeta_1\zeta_2 \cxl $.
For the remaining resolutions we will simply start with the greatest common divisor without further comment.
\begin{align*}
	\zeta_0\zeta_1\zeta_2 \cxl 
		&\xmapsto{\text{\ref{red:1}}} \xi \cxl  \\
	\zeta_0\zeta_1\zeta_2 \cxl 
		&\xmapsto{\text{\ref{red:2}}} (1-\kappa)\zeta_0^2\zeta_2^2 \cl + e^2\zeta_0\zeta_2 \\
		&\xmapsto{\text{\ref{red:3}}} \xi \cxl 
\end{align*}

\item \ref{red:1} and \ref{red:3}:
\begin{align*}
	\zeta_0^2\zeta_1\zeta_2^2 \cl
		&\xmapsto{\text{\ref{red:1}}} \xi\zeta_0\zeta_2 \cl \\
	\zeta_0^2\zeta_1\zeta_2^2 \cl
		&\xmapsto{\text{\ref{red:3}}} \xi\zeta_1 \cxl  + e^2\zeta_0\zeta_1\zeta_2 \\
		&\xmapsto{\text{\ref{red:1}}} \xi\zeta_1 \cxl  + e^2\xi \\
		&\xmapsto{\text{\ref{red:2}}} \xi\zeta_0\zeta_2 \cl
\end{align*}

\item \ref{red:1} and \ref{red:4}:
\begin{align*}
	\zeta_0\zeta_1\zeta_2^2 \cxw
		&\xmapsto{\text{\ref{red:1}}} \xi\zeta_2 \cxw \\
	\zeta_0\zeta_1\zeta_2^2 \cxw
		&\xmapsto{\text{\ref{red:4}}} (1-\kappa)\zeta_0^3\zeta_1 \cw  + e^2\zeta_0\zeta_1 \cxl  \\
		&\xmapsto{\text{\ref{red:2}}} (1-\kappa)\zeta_0^3\zeta_1 \cw 
			- e^2\zeta_0^2\zeta_2 \cl + e^4\zeta_0 \\
		&\xmapsto{\text{\ref{red:5}}} \xi\zeta_2 \cxl 
\end{align*}

\item \ref{red:1} and \ref{red:5}:
\begin{align*}
	\zeta_0^3\zeta_1\zeta_2 \cw 
		&\xmapsto{\text{\ref{red:1}}} \xi\zeta_0^2 \cw  \\
	\zeta_0^3\zeta_1\zeta_2 \cw 
		&\xmapsto{\text{\ref{red:5}}} \xi\zeta_2^2 \cxw - e^2\zeta_0^2\zeta_2^2 \cl + e^4\zeta_0\zeta_2 \\
		&\xmapsto{\text{\ref{red:4}}} \xi\zeta_0^2 \cw  + e^2\xi \cxl 
			 - e^2\zeta_0^2\zeta_2^2 \cl + e^4\zeta_0\zeta_2 \\
		&\xmapsto{\text{\ref{red:3}}} \xi\zeta_0^2 \cw 
\end{align*}

\item \ref{red:1} and \ref{red:6}:
\begin{align*}
	\zeta_0^4\zeta_1\zeta_2 \cl\cw 
		&\xmapsto{\text{\ref{red:1}}} \xi\zeta_0^3 \cl\cw  \\
	\zeta_0^4\zeta_1\zeta_2 \cl\cw 
		&\xmapsto{\text{\ref{red:6}}} \xi\zeta_1\zeta_2 \cxl \cxw 
			+ e^2\zeta_0^2\zeta_1\zeta_2 \cl \cxl  - e^2\zeta_0\zeta_1\zeta_2^2 \cxw \\
		&\xmapsto{\text{\ref{red:1}}} \xi\zeta_1\zeta_2 \cxl \cxw 
			+ e^2\xi\zeta_0 \cl \cxl  - e^2\xi\zeta_2 \cxw \\
		&\xmapsto{\text{\ref{red:2}}} \xi\zeta_0\zeta_2^2 \cl\cxw  
			+ e^2\xi\zeta_0 \cl \cxl  \\
		&\xmapsto{\text{\ref{red:4}}} \xi\zeta_0^3 \cl\cw 
\end{align*}

\item \ref{red:2} and \ref{red:3}: In this case, the greatest common divisor is 1, so the resolution is immediate.

\item \ref{red:2} and \ref{red:4}: The greatest common divisor is 1.

\item \ref{red:2} and \ref{red:5}:
\begin{align*}
	\zeta_0^3\zeta_1 \cxl \cw 
		&\xmapsto{\text{\ref{red:2}}} (1-\kappa)\zeta_0^4\zeta_2 \cl\cw  + e^2\zeta_0^3 \cw  \\
		&\xmapsto{\text{\ref{red:6}}} \xi\zeta_2 \cxl \cxw - e^2\zeta_0^2\zeta_2 \cl \cxl 
			+ e^2\zeta_0\zeta_2^2 \cxw + e^2\zeta_0^3 \cw  \\
		&\xmapsto{\text{\ref{red:4}}} \xi\zeta_2 \cxl \cxw - e^2\zeta_0^2\zeta_2 \cl \cxl 
			+ e^4\zeta_0 \cxl  \\
	\zeta_0^3\zeta_1 \cxl \cw 
		&\xmapsto{\text{\ref{red:5}}} \xi\zeta_2 \cxl \cxw 
			- e^2\zeta_0^2\zeta_2 \cl \cxl  + e^4\zeta_0 \cxl 
\end{align*}

\item \ref{red:2} and \ref{red:6}: The greatest common divisor is 1.

\item \ref{red:3} and \ref{red:4}: 
\begin{align*}
	\zeta_0^2\zeta_2^2 \cl\cxw
		&\xmapsto{\text{\ref{red:3}}} \xi \cxl \cxw + e^2\zeta_0\zeta_2 \cxw \\
	\zeta_0^2\zeta_2^2 \cl\cxw
		&\xmapsto{\text{\ref{red:4}}} (1-\kappa)\zeta_0^4 \cl\cw  + e^2\zeta_0^2 \cl \cxl  \\
		&\xmapsto{\text{\ref{red:6}}} \xi \cxl \cxw + e^2\zeta_0\zeta_2 \cxw
\end{align*}

\item \ref{red:3} and \ref{red:5}: 
\begin{align*}
	\zeta_0^3\zeta_1\zeta_2^2 \cl\cw 
		&\xmapsto{\text{\ref{red:3}}} \xi\zeta_0\zeta_1 \cxl \cw  + e^2\zeta_0^2\zeta_1\zeta_2 \cw  \\
		&\xmapsto{\text{\ref{red:1}}} \xi\zeta_0\zeta_1 \cxl \cw  + e^2\xi\zeta_0 \cw  \\
		&\xmapsto{\text{\ref{red:2}}} \xi\zeta_0^2\zeta_2 \cl\cw  \\
	\zeta_0^3\zeta_1\zeta_2^2 \cl\cw 
		&\xmapsto{\text{\ref{red:5}}} \xi\zeta_2^3 \cl\cxw - e^2\zeta_0^2\zeta_2^3 \cl^2
			+ e^4\zeta_0\zeta_2^2 \cl \\
		&\xmapsto{\text{\ref{red:4}}} \xi\zeta_0^2\zeta_2 \cl\cw  + e^2\xi\zeta_2 \cl \cxl 
			- e^2\zeta_0^2\zeta_2^3 \cl^2 + e^4\zeta_0\zeta_2^2 \cl \\
		&\xmapsto{\text{\ref{red:3}}} \xi\zeta_0^2\zeta_2 \cl\cw  
\end{align*}

\item \ref{red:3} and \ref{red:6}: 
\begin{align*}
	\zeta_0^4\zeta_2^2 \cl\cw 
		&\xmapsto{\text{\ref{red:3}}} \xi\zeta_0^2 \cxl \cw  + e^2\zeta_0^3\zeta_2 \cw  \\
	 \zeta_0^4\zeta_2^2 \cl\cw 
		&\xmapsto{\text{\ref{red:6}}} \xi\zeta_2^2 \cxl \cxw + e^2\zeta_0^2\zeta_2^2 \cl \cxl 
			- e^2\zeta_0\zeta_2^3 \cxw \\
		&\xmapsto{\text{\ref{red:4}}} \xi\zeta_0^2 \cxl \cw  + e^2\xi \cxl ^2
			+ e^2\zeta_0^2\zeta_2^2 \cl \cxl 
			+ e^2\zeta_0^3\zeta_2 \cw  - e^4\zeta_0\zeta_2 \cxl  \\
		&\xmapsto{\text{\ref{red:3}}} \xi\zeta_0^2 \cxl \cw  + e^2\zeta_0^3\zeta_2 \cw 
\end{align*}

\item \ref{red:4} and \ref{red:5}: The greatest common divisor is 1.

\item \ref{red:4} and \ref{red:6}: The greatest common divisor is 1.

\item \ref{red:5} and \ref{red:6}:
\begin{align*}
	\zeta_0^4\zeta_1 \cl\cw 
		&\xmapsto{\text{\ref{red:5}}} \xi\zeta_0\zeta_2 \cl\cxw - e^2\zeta_0^3\zeta_2 \cl^2
			+ e^4\zeta_0^2 \cl \\
	\zeta_0^4\zeta_1 \cl\cw 
		&\xmapsto{\text{\ref{red:6}}} \xi\zeta_1 \cxl \cxw + e^2\zeta_0^2\zeta_1 \cl \cxl 
			- e^2\zeta_0\zeta_1\zeta_2 \cxw \\
		&\xmapsto{\text{\ref{red:1}}} \xi\zeta_1 \cxl \cxw + e^2\zeta_0^2\zeta_1 \cl \cxl 
			- e^2\xi \cxw \\
		&\xmapsto{\text{\ref{red:2}}} \xi\zeta_0\zeta_2 \cl\cxw
			- e^2\zeta_0^3\zeta_2 \cl^2 + e^4\zeta_0^2 \cl
\end{align*}

\end{itemize}
This resolves all of the ambiguities of the reduction system, completing the proof of
Proposition~\ref{prop:freeness}.

\bibliography{Bibliography}{}
\bibliographystyle{amsplain}

\end{document}